\documentclass[3p,times]{elsarticle}

\journal{Journal of Computational and Applied Mathematics}

\usepackage{amsmath}
\usepackage{mathrsfs} % needed for \mathscr{D}
\usepackage{tikz}
\usepackage{gnuplot-lua-tikz}
\newcommand{\domD}{\mathscr{D}}
\renewcommand{\pi}{\piup}
\renewcommand{\Re}{\operatorname{Re}}
\renewcommand{\Im}{\operatorname{Im}}
\DeclareMathOperator{\arsinh}{arsinh}
\DeclareMathOperator{\OO}{O}
\makeatletter
\newif\ifkp@upRm% is used in the .fd-file of jkp
\DeclareSymbolFont{Letters}{OML}{jkp}{m}{it}
\DeclareMathSymbol{\uppartial}{\mathord}{Letters}{128}
\makeatother
\DeclareMathOperator{\E}{e}
\DeclareMathOperator{\I}{i}

\newcommand{\diff}{\mathrm{d}}
\newcommand{\D}{\,\mathrm{d}}

\newdefinition{definition}{Definition}[section]
\newtheorem{theorem}{Theorem}[section]
\newtheorem{lemma}[theorem]{Lemma}
\newtheorem{proposition}[theorem]{Proposition}

\newtheorem{remark}{Remark}[section]

\newproof{proof}{Proof}
\numberwithin{equation}{section}

\begin{document}

\begin{frontmatter}

\title{DE-Sinc approximation for unilateral rapidly decreasing functions and its computable error bound~\tnoteref{mytitlenote}}
\tnotetext[mytitlenote]{This work was partially supported by the JSPS
Grant-in-Aid for Scientific Research (C) JP23K03218.}

%\author{Tomoaki Okayama\fnref{myfootnote}}
\author[HCU]{Tomoaki Okayama\corref{cor1}}
\cortext[cor1]{Corresponding author}
\affiliation[HCU]{organization={Hiroshima City University},
addressline={3-4-1, Ozuka-higashi, Asaminami-ku},
city={Hiroshima},
postcode={731-3194},
country={Japan}}
\ead{okayama@hiroshima-cu.ac.jp}

%\address{Graduate School of Information Sciences, Hiroshima City University,
%3-4-1, Ozuka-higashi, Asaminami-ku, Hiroshima 731-3194, Japan}
%\ead{okayama@hiroshima-cu.ac.jp}

%\author[HCU]{Yuito Kuwashita}
%\author[HCU]{Ao Kondo}




\begin{abstract}
The Sinc approximation is highly effective for functions that decay
rapidly at both ends of the real axis.
For unilateral rapidly decreasing functions, which decay algebraically
as $t\to-\infty$ and exponentially as $t\to\infty$, an appropriate
variable transformation is required.
Existing single-exponential transformations for this class of functions
yield only root-exponential convergence, even when improved transformations
are employed.
This paper develops a double-exponential (DE)-Sinc approximation based on
the transformation
$t=2\sinh(\log(\log(1+\exp(\pi\sinh x))))$,
which was previously introduced for numerical integration.
The main contribution is a rigorous and computable error bound of order
$\OO(\exp(-cn/\log n))$ with a constant explicitly expressed in terms of
the problem parameters.
Under the stated assumptions, the resulting approximation achieves
almost exponential convergence and is suitable for computation with
guaranteed accuracy.
Numerical examples satisfying these assumptions confirm the predicted
convergence behavior and the validity of the derived error bound.
\end{abstract}

\begin{keyword}
Sinc approximation
\sep double-exponential transformation
\sep unilateral rapidly decreasing function
\sep computation with guaranteed accuracy
\MSC[2010] 65D05, 65D15
\end{keyword}

\end{frontmatter}

\section{Introduction}
\label{sec:intro}

The Sinc approximation on the real axis is expressed as
\begin{equation}
\label{eq:Sinc-approximation}
 F(x) \approx \sum_{k=-M}^N F(kh)S(k,h)(x),\quad x\in\mathbb{R},
\end{equation}
where $h$ is the mesh size,
$M$ and $N$ are the truncation numbers,
and $S(k, h)$ is the so-called Sinc function, defined by
\[
 S(k,h)(x)=
\begin{cases}
\dfrac{\sin[\pi(x - kh)/h]}{\pi(x - kh)/h} & (x\neq kh), \\
1 & (x = kh).
\end{cases}
\]
The approximation~\eqref{eq:Sinc-approximation} is particularly
effective when $F$ is analytic in a strip around the real axis and
decays rapidly as $x\to\pm\infty$.
When the target function $f$ does not have this decay property,
an appropriate variable transformation can be used to produce a
rapidly decreasing transformed function.
Stenger~\cite{stenger93:_numer,Stenger} classified the following five
typical cases according to the target interval $(a,b)$ and the decay
behavior of $f$:
\begin{enumerate}
 \item the interval $(a,b)$ is finite,
 \item $(a, b)=(0,\infty)$ and $|f(t)|$ decays algebraically as $t\to\infty$,
 \item $(a, b)=(0,\infty)$ and $|f(t)|$ decays exponentially as $t\to\infty$,
 \item $(a, b)=(-\infty,\infty)$ and $|f(t)|$ decays algebraically as $t\to\pm\infty$,
 \item $(a, b)=(-\infty,\infty)$ and $|f(t)|$ decays algebraically as $t\to-\infty$ and exponentially as $t\to\infty$.
\end{enumerate}
In this paper, a function in case~5 is called a unilateral rapidly
decreasing function.
Such a function appears, for example, as the exact solution in
Example~7.4 of ODE-IVP-PACK~\cite{StengerODEIVP}.
For the five cases listed above, Stenger proposed the transformations
$t=\psi_i(x)$ $(i=1,\,\ldots,\,5)$ given by
\begin{align*}
 t &= \psi_1(x)
    = \frac{b - a}{2}\tanh\left(\frac{x}{2}\right)+\frac{b + a}{2},\\
 t &= \psi_2(x) = \E^x,\\
 t &= \psi_3(x) = \arsinh(\E^x),\\
 t &= \psi_4(x) = \sinh x,\\
 t &= \psi_5(x) = \sinh(\log(\arsinh(\E^x))).
\end{align*}
Combining the Sinc approximation~\eqref{eq:Sinc-approximation} with
$t=\psi_i(x)$ gives
\begin{align}
  f(\psi_i(x))&\approx \sum_{k=-M}^N f(\psi_i(kh))S(k,h)(x),
\quad x\in \mathbb{R},\nonumber \\
\intertext{which is equivalent to}
 f(t)&\approx \sum_{k=-M}^N f(\psi_i(kh))S(k,h)(\psi_i^{-1}(t)),
\quad t\in (a,b).
\label{eq:psi-Sinc-approximation}
\end{align}
For all five cases, this formula can attain root-exponential
convergence of order $\OO(\exp(-c\sqrt{n}))$,
where $n=\max\{M, N\}$.
Owing to this high efficiency, numerical methods based on
\eqref{eq:psi-Sinc-approximation} have been developed in a variety of
areas~\cite{lund92:_sinc_method_quadr_differ_equat,stenger93:_numer,stenger00:_summar,Stenger},
which are collectively called Sinc numerical methods.

For cases $i=1,\,\ldots,\,4$, subsequent studies derived explicit
error bounds~\cite{OkayamaSinc,Okayama-et-al} of the form
\begin{equation*}
 \sup_{t\in (a,b)}\left|
 f(t) - \sum_{k=-M}^N f(\psi_i(kh))S(k,h)(\psi_i^{-1}(t))
\right|
\leq C_i \sqrt{n}\E^{-\sqrt{\pi d \mu n}},
\end{equation*}
where the positive parameters $d$ and $\mu$ describe the analyticity
and decay properties of the transformed function, respectively.
The constant $C_i$ is explicitly expressed in terms of the problem
parameters, so the right-hand side provides a computable rigorous
error bound suitable for algorithms with guaranteed accuracy.
Improved transformations for cases $i=3$ and $i=5$ were later
proposed~\cite{OkaShinKatsu,okayama21:_improv_sinc} as
\begin{align*}
 t &= \tilde{\psi}_3(x) = \log(1 + \E^x),\\
 t &= \tilde{\psi}_5(x) = 2\sinh(\log(\log(1 + \E^x))).
\end{align*}
These transformations allow a larger value of $d$ than the original
transformations $\psi_3$ and $\psi_5$, thereby improving the error
bound.  Nevertheless, the asymptotic convergence order remains
root-exponential.

For cases $i=1,\,\ldots,\,4$, substantially faster convergence is
obtained by using the transformations
\begin{align*}
 t &= \phi_1(x)
    = \frac{b-a}{2}\tanh\left(\frac{\pi}{2}\sinh x\right) +\frac{b+a}{2},\\
 t &= \phi_2(x) = \E^{(\pi/2)\sinh x},\\
 t &= \phi_3(x) = \log(1 + \E^{(\pi/2)\sinh x}),\\
 t &= \phi_4(x) = \sinh\left((\pi/2)\sinh x\right).
\end{align*}
These are called the double-exponential (DE) transformations because the
transformed functions decay double exponentially as $x\to\pm\infty$.
They were originally introduced for numerical
integration~\cite{muhammad03:_doubl,takahasi74:_doubl}.
By contrast, transformations that produce single-exponential decay are
called the single-exponential (SE) transformations.
The combination of a DE transformation with
\eqref{eq:Sinc-approximation}, called the DE-Sinc approximation, can
attain almost exponential convergence of order
$\OO(\exp(-c n/\log n))$~\cite{tanaka07:_class_of_funct}.
Its optimality in a certain sense has also been
established~\cite{sugihara02:_near_optim}, and DE-Sinc numerical methods
have been developed in various
areas~\cite{Mori,mori01:_doubl_expon,SugiharaMatsuo}.

Recent research continues to develop approximation methods based on
variable transformations and Sinc techniques in a range of settings.
Examples include log-orthogonal approximations for slowly decaying
functions on semi-infinite intervals~\cite{ChenShen2023}, DE-Sinc
collocation for high-order boundary value problems~\cite{QiuEtAl2022},
transformed Sinc approximation for periodic functions on finite
intervals~\cite{Ogata2025}, and Sinc quadrature with sharp error
estimates for Stieltjes functions on finite and infinite
intervals~\cite{BraessHackbusch2026}.
Although these problems differ from the unilateral-decay problem
considered here, they demonstrate the continuing relevance of
transformation-based approximation and Sinc techniques in contemporary
numerical analysis.

For the DE-Sinc approximations based on $\phi_i$ in cases
$i=1,\,\ldots,\,4$, rigorous error bounds with explicit constants
were obtained in~\cite{OkayamaSinc} as
\begin{equation}
 \sup_{t\in (a,b)}\left|
 f(t) - \sum_{k=-M}^N f(\phi_i(kh))S(k,h)(\phi_i^{-1}(t))
\right|
\leq C_i \E^{-\pi d n/\log(2 d n/\mu)},
\label{eq:DE-Sinc-error-bound}
\end{equation}
where, for case $i=3$, the improved DE transformation
\[
 t = \tilde{\phi}_3(x) = \log(1 + \E^{\pi\sinh x})
\]
was used.

In contrast to cases $i=1,\,\ldots,\,4$, a DE-Sinc approximation with
a computable error bound had not been established for case~5.
This paper fills this gap by employing the DE transformation
\[
 t = \phi_5(x) = 2\sinh(\log(\log(1 + \E^{\pi\sinh x}))),
\]
which was previously proposed for numerical
integration~\cite{okayama22:_doubl}.
Our main result is a rigorous error bound of the
form~\eqref{eq:DE-Sinc-error-bound}, with a constant explicitly
expressed in terms of the problem parameters.
It establishes almost exponential convergence for the DE-Sinc
approximation of unilateral rapidly decreasing functions and enables
computation with guaranteed accuracy.
This improvement is also substantial in terms of computational cost.
Apart from multiplicative constants and lower-order factors, attaining
an accuracy $\varepsilon$ requires
$n=\OO((\log(1/\varepsilon))^2)$ for the SE-Sinc approximation, whereas
the proposed DE-Sinc approximation requires only
$n=\OO(\log(1/\varepsilon)\log\log(1/\varepsilon))$.

The proof of the main result is not obtained by simply substituting $\phi_5$ into existing
DE-Sinc estimates: the nested logarithms in $\phi_5$, its branch
structure in the complex plane, and the asymmetric behavior at the two
ends of the real axis require new estimates on the strip domain.
These estimates yield explicit bounds for both the discretization and
truncation errors and are the main technical ingredients of the present
analysis.
Since the Sinc approximation is a basic component of quadrature and
collocation methods, the present theorem may also serve as an analytical
foundation for future guaranteed-accuracy DE-Sinc algorithms for
problems involving unilateral decay.
Numerical examples confirm both the predicted convergence behavior and
the validity of the derived error bound.

The remainder of this paper is organized as follows.
Section~\ref{sec:main_theorem} summarizes the existing approximation
formulas and states the new error estimate.
Section~\ref{sec:numer} presents numerical examples, and
Section~\ref{sec:proofs} proves the new result.
Section~\ref{sec:conclusion} presents concluding remarks.

\section{Summary of existing and new theorems}
\label{sec:main_theorem}

Section~\ref{subsec:Stenger-result} describes existing theorems,
and Section~\ref{subsec:our-result} describes a new theorem
(the main result of this paper).
To state these theorems,
some notation is introduced.
Let $\domD_d$ be a strip domain
defined by $\domD_d=\{\zeta\in\mathbb{C}:|\Im \zeta|< d\}$
for $d>0$.
Furthermore, let
$\domD_d^{-}=\{\zeta\in\domD_d :\Re\zeta< 0\}$
and
$\domD_d^{+}=\{\zeta\in\domD_d :\Re\zeta\geq 0\}$.
When $\psi_5$, $\tilde{\psi}_5$, and $\phi_5$ are evaluated at
complex arguments, they are understood as their analytic
continuations from the real axis to $\domD_d$.
The branches used for $\phi_5$ are described explicitly in
Section~\ref{sec:proofs}.
Throughout this paper, $n$ denotes a positive integer that controls
the truncation numbers $M$ and $N$.
The definitions of $M$ and $N$ in the following theorems ensure that
$n=\max\{M,N\}$,
and the number of sampling points is $M+N+1$.

\subsection{Existing results}
\label{subsec:Stenger-result}

An error analysis of the Sinc approximation
combined with $t=\psi_5(x)$ can be expressed as
the following theorem,
which is essentially the same as
the existing theorem~\cite[Theorem 1.5.13]{Stenger}.

\begin{theorem}
\label{thm:Stenger-2}
Assume that $f$ is analytic in $\psi_5(\domD_d)$ with $0<d\leq\pi/2$,
and that there exist positive constants
$K_{+}$, $K_{-}$, $\alpha$ and $\beta$
such that
\begin{align}
 |f(z)|&\leq K_{-} |z|^{-\alpha}
\label{eq:f-bound-minus}
\intertext{holds for all $z\in\psi_5(\domD_d^{-})$, and}
 |f(z)|&\leq K_{+} |\E^{-z}|^{2\beta}\nonumber
\end{align}
holds for all $z\in\psi_5(\domD_d^{+})$.
Let $\mu = \min\{\alpha,\beta\}$,
let $n$ be a positive integer,
let $M$ and $N$ be defined by
\begin{equation}
M = \left\lceil\frac{\mu}{\alpha}n\right\rceil,
\quad
N = \left\lceil\frac{\mu}{\beta }n\right\rceil,
\label{eq:Def-MN}
\end{equation}
and let $h$ be defined by
\begin{equation}
h = \sqrt{\frac{\pi d}{\mu n}}.
\label{eq:Def-h}
\end{equation}
Then, there exists a constant $C$ independent of $n$ such that
\[
 \sup_{t\in \mathbb{R}}\left|
 f(t)
- \sum_{k=-M}^N f(\psi_5(kh))S(k,h)(\psi_5^{-1}(t))
\right|
\leq C \sqrt{n}\E^{-\sqrt{\pi d \mu n}}.
\]
\end{theorem}

For the improved transformation $\tilde{\psi}_5$,
the following theorem was established.

\begin{theorem}[Okayama and Shiraishi~{\cite[Theorem~2]{okayama21:_improv_sinc}}]
\label{thm:Stenger-4}
Assume that $f$ is analytic in $\tilde{\psi}_5(\domD_d)$ with $0<d<\pi$,
and that there exist positive constants
$K_{+}$, $K_{-}$, $\alpha$ and $\beta$
such that~\eqref{eq:f-bound-minus}
holds for all $z\in\tilde{\psi}_5(\domD_d^{-})$,
and
\begin{equation}
 |f(z)|\leq K_{+} |\E^{-z}|^{\beta}
\label{eq:f-bound-plus}
\end{equation}
holds for all $z\in\tilde{\psi}_5(\domD_d^{+})$.
Let $\mu = \min\{\alpha,\beta\}$,
let $n$ be a positive integer,
let $M$ and $N$ be defined by~\eqref{eq:Def-MN},
and let $h$ be defined by~\eqref{eq:Def-h}.
Then, we have
\[
 \sup_{t\in \mathbb{R}}\left|
 f(t)
- \sum_{k=-M}^N f(\tilde{\psi}_5(kh))S(k,h)(\tilde{\psi}_5^{-1}(t))
\right|
\leq
\left[
\frac{2 \tilde{C}_{\mathrm{D}}}{\pi d (1 - \E^{-2\sqrt{\pi d \mu}})}
+ \tilde{C}_{\mathrm{T}}\sqrt{\frac{\mu}{\pi d}}
\right]
 \sqrt{n}\E^{-\sqrt{\pi d \mu n}},
\]
where $\tilde{C}_{\mathrm{D}}$ and $\tilde{C}_{\mathrm{T}}$
are constants defined by
\begin{align*}
 \tilde{C}_{\mathrm{D}}
&=\frac{K_{-}}{\alpha}
\left[\frac{\E}{(1 - \log 2)(\E - 1)\cos(d/2)}\right]^{\alpha}
+\frac{K_{+}}{\beta}\left[\frac{\E^{1/\log 2}}{\cos(d/2)}\right]^{\beta},\\
\tilde{C}_{\mathrm{T}}
&=\frac{K_{-}}{\alpha}\left(\frac{1}{1 - \log 2}\right)^{\alpha}
+\frac{K_{+}}{\beta}\left(\E^{1/\log 2}\right)^{\beta}.
\end{align*}
\end{theorem}

The difference between Theorems~\ref{thm:Stenger-2}
and~\ref{thm:Stenger-4} was discussed~\cite{okayama21:_improv_sinc},
which led to the conclusion that
the parameters $d$ and $\mu$ in Theorem~\ref{thm:Stenger-4}
may be taken larger
than those in Theorem~\ref{thm:Stenger-2}.
However, the convergence rates are essentially the same,
in the sense that both rates are of
root-exponential order: $\OO(\exp(-c\sqrt{n}))$.

\subsection{New results}
\label{subsec:our-result}

To improve the convergence rate significantly,
this study proposes the DE transformation
$t=\phi_5(x)$,
and establishes the following theorem.
The proof is provided in Section~\ref{sec:proofs}.

\begin{theorem}
\label{thm:New1}
Let $L$ be the positive constant defined by
$L=\log(\E/(\E-1))$, and let $d_L$ be the positive constant defined by
\begin{equation}
 d_L = \arccos\left(\sqrt{\frac{2}{1+\sqrt{1 + (2\pi/L)^2}}}\right).
\label{eq:d_L}
\end{equation}
%Note that $0<d_L<\pi/2$.
Assume that $f$ is analytic in $\phi_5(\domD_d)$
with $0 < d < d_L$,
and that there exist positive constants
$K_{+}$, $K_{-}$, $\alpha$ and $\beta$
such that~\eqref{eq:f-bound-minus}
holds for all $z\in\phi_5(\domD_d^{-})$,
and~\eqref{eq:f-bound-plus}
holds for all $z\in\phi_5(\domD_d^{+})$.
Let $\mu = \min\{\alpha,\beta\}$,
let $n$ be a positive integer with $n\geq \max\{\mu\E,\alpha,\beta\}/(2d)$,
let $M$ and $N$ be defined by
\begin{equation}
 M = n - \left\lfloor
\frac{\log(\alpha/\mu)}{h}
\right\rfloor,
\quad
 N = n - \left\lfloor
\frac{\log(\beta/\mu)}{h}
\right\rfloor,
\label{eq:Def-MN-DE}
\end{equation}
and let $h$ be defined by
\begin{equation}
 h = \frac{\log(2 d n/\mu)}{n}.
\label{eq:Def-h-DE}
\end{equation}
Then, there exists a constant $C$ independent of $n$ such that
\begin{equation}
 \sup_{t\in \mathbb{R}}\left|
 f(t)
- \sum_{k=-M}^N f(\phi_5(kh))S(k,h)(\phi_5^{-1}(t))
\right|\leq C \E^{-\pi d n /\log(2 d n/\mu)}.
\label{eq:DE-approx-error}
\end{equation}
More precisely, define
\begin{align}
 T_{L,d}&=\frac{\pi\cos^2d-L\sin d}{2\sin d},
 \label{eq:T-L-d}\\
 \tau_{L,d}&=\frac{T_{L,d}}{1+\sqrt{T_{L,d}}},
 \label{eq:tau-L-d}\\
 a_{L,d}&=\frac{\pi}{2}\sin d
 \sqrt{1+\left( \frac{L + 2\tau_{L,d}}{\pi\cos d}
 \right)^2}.
 \label{eq:a-L-d}
\end{align}
Then, the inequality~\eqref{eq:DE-approx-error} holds with
\[
 C = \frac{1}{\pi d}\left[
\frac{2 C_{\mathrm{D}}}{\pi (1 - \E^{-\pi \mu \E})\cos d}
+ C_{\mathrm{T}}
\right],
\]
where ${C}_{\mathrm{D}}$ and ${C}_{\mathrm{T}}$
are constants defined by
\begin{align}
  {C}_{\mathrm{D}}
&=\frac{K_{-}}{\alpha}
\left[\frac{\E}{(1 - \log 2)(\E - 1)\cos a_{L,d}}\right]^{\alpha}
+\frac{K_{+}}{\beta}
\left[\frac{\E^{1/\log 2}}{\cos((\pi/2)\sin d)}\right]^{\beta},
\label{eq:C_D}\\
 {C}_{\mathrm{T}}
&=K_{-}\left(\frac{\E^{\pi/2}}{1 - \log 2}\right)^{\alpha}
+K_{+}\left(\E^{(\pi/2)+(1/\log 2)}\right)^{\beta}.
\label{eq:C_T}
\end{align}
\end{theorem}

From Theorem~\ref{thm:New1},
we conclude that the proposed method achieves
almost exponential convergence, which is substantially faster than
root-exponential convergence.  Their respective orders are
\[
\OO(\exp(-c n/\log n))
\quad\text{and}\quad
\OO(\exp(-c\sqrt{n})).
\]

Table~\ref{tbl:comparison-theorems} summarizes the three approximation
results presented in this section.
The first two results use the SE transformations and attain
root-exponential convergence, whereas the new result uses the DE
transformation and attains almost exponential convergence.
Because $M+N+1$ is asymptotically proportional to $n$, the
sample-complexity row gives the asymptotic number of sampling points
required to attain an accuracy $\varepsilon$.
In the last row, ``explicit'' means that every constant in the error
bound is given in a computable form.

\begin{table}[!ht]
\centering
\caption{Comparison of the approximation results in Section~\ref{sec:main_theorem}.}
\label{tbl:comparison-theorems}
\small
\renewcommand{\arraystretch}{1.25}
\setlength{\tabcolsep}{3pt}
\begin{tabular}{@{}p{0.20\textwidth}p{0.23\textwidth}p{0.23\textwidth}p{0.25\textwidth}@{}}
\hline
& Theorem~\ref{thm:Stenger-2}
& Theorem~\ref{thm:Stenger-4}
& Theorem~\ref{thm:New1} \\
\hline
Transformation
& $\psi_5$ (SE)
& $\tilde{\psi}_5$ (SE)
& $\phi_5$ (DE) \\
Strip parameter
& $0<d\leq\pi/2$
& $0<d<\pi$
& $0<d<d_L$ \\
Mesh size $h$
& $\sqrt{\pi d/(\mu n)}$
& $\sqrt{\pi d/(\mu n)}$
& $\log(2dn/\mu)/n$ \\
Error order
& $\OO(\sqrt{n}\E^{-\sqrt{\pi d\mu n}})$
& $\OO(\sqrt{n}\E^{-\sqrt{\pi d\mu n}})$
& $\OO(\E^{-\pi dn/\log(2dn/\mu)})$ \\
Sample complexity
& $\OO((\log(1/\varepsilon))^2)$
& $\OO((\log(1/\varepsilon))^2)$
& $\OO(\log(1/\varepsilon)\log\log(1/\varepsilon))$ \\
Explicit error constant
& No
& Yes
& Yes \\
\hline
\end{tabular}
\end{table}

%\clearpage

\section{Numerical examples}
\label{sec:numer}

This section presents numerical results.
All programs were written in C with double-precision floating-point
arithmetic.
Let us consider the following three unilateral rapidly decreasing functions:
\begin{align}
 f_1(t) &= \sinh\left(\frac{t + \sqrt{4 + t^2}}{4}\right)
\E^{-t - \sqrt{t^2 + 4}}, \label{eq:f_1}\\
 f_2(t) &= \frac{1}{\sqrt{1 + (t/2)^2} + 1 - (t/2)}
\E^{-(t/2)-\sqrt{1 + (t/2)^2}}, \label{eq:f_2}\\
 f_3(t) &= \frac{1}{\sqrt{1+t^2}(1+\E^t)}. \label{eq:f_3}
\end{align}
The functions $f_2(t)$ and $f_3(t)$
were taken from the previous studies~\cite{okayama21:_improv_sinc,StengerODEIVP}.
The functions $f_1$ and $f_2$ satisfy the assumptions
of Theorems~\ref{thm:Stenger-2}--\ref{thm:New1} with
the parameters shown in Tables~\ref{tbl:parameters-1}
and~\ref{tbl:parameters-2}
($K_{-}$ and $K_{+}$ for Theorem~\ref{thm:Stenger-2}
are not shown because their existence is evident,
and they are not required for computation).
The function $f_3$ satisfies the assumptions of
Theorems~\ref{thm:Stenger-2} and~\ref{thm:Stenger-4} with the
parameters shown in Table~\ref{tbl:parameters-3}, whereas
Theorem~\ref{thm:New1} is not applicable because no positive value of
$d$ satisfies its analyticity assumption for this function.
For comparison, we nevertheless applied the DE-Sinc approximation to $f_3$
with $\alpha=\beta=1$ and $d=1.17$, where the value of $d$ was chosen
only to match that used for $f_1$ and $f_2$ and has no theoretical
justification.  Accordingly, only the observed error, and not the error
bound from Theorem~\ref{thm:New1}, is plotted for this computation.
The errors were investigated on the following 402 points
\[
 t = \pm 2^{i}, \quad i = -50,\,-49.5, -49,\,\ldots,\,49.5,\,50,
\]
and $t=0$ (403 points in total),
and maximum error among these points was plotted on the graph.
The results are shown in
Figs.~\ref{fig:example1}--\ref{fig:example3}.
For $f_1$ and $f_2$, we observe that
the DE-Sinc approximation (proposed formula)
from Theorem~\ref{thm:New1}
converges faster than the formulas from Theorems~\ref{thm:Stenger-2}
and~\ref{thm:Stenger-4}.
Furthermore, the error bounds from Theorems~\ref{thm:Stenger-4}
and~\ref{thm:New1} include the corresponding observed errors,
reflecting their respective theoretical convergence rates:
$\OO(\exp(-c\sqrt{n}))$ and
$\OO(\exp(-cn / \log n))$.
For $f_3$, the observed errors obtained using
Theorems~\ref{thm:Stenger-2} and~\ref{thm:Stenger-4} are nearly the
same, although the latter is slightly smaller.  The exploratory DE-Sinc
approximation does not exhibit the characteristic
$\OO(\exp(-cn/\log n))$ convergence; instead, its observed convergence
is similar to, and appears slightly faster than, those of the two
SE-Sinc approximations.  Related behavior for functions satisfying
SE-Sinc assumptions but not the standard DE-Sinc assumptions has been
analyzed in~\cite{OkayamaEtAl2013DESincSE}, which may provide a useful
starting point for studying the present case.  A rigorous error analysis is left for future work.

\begin{table}
\centering
\caption{Parameters for a function $f_1(t)$ in~\eqref{eq:f_1}.}
\label{tbl:parameters-1}
\begin{tabular}{c|lllrl}
\hline
& \multicolumn{1}{c}{$d$} & \multicolumn{1}{c}{$\alpha$} & \multicolumn{1}{c}{$\beta$} & \multicolumn{1}{c}{$K_{-}$} & \multicolumn{1}{c}{$K_{+}$} \\
\hline
Theorem~\ref{thm:Stenger-2} & 1.5 & 1        & 0.75    &        &       \\
Theorem~\ref{thm:Stenger-4} & 3   & 1        & 1.5     & 159    & 5.73  \\
Theorem~\ref{thm:New1}      & 1.17& 1        & 1.5     &  34    & 3.39  \\
\hline
\end{tabular}
\caption{Parameters for a function $f_2(t)$ in~\eqref{eq:f_2}.}
\label{tbl:parameters-2}
\begin{tabular}{c|lllll}
\hline
& \multicolumn{1}{c}{$d$} & \multicolumn{1}{c}{$\alpha$} & \multicolumn{1}{c}{$\beta$} & \multicolumn{1}{c}{$K_{-}$} & \multicolumn{1}{c}{$K_{+}$} \\
\hline
Theorem~\ref{thm:Stenger-2} & 1.5 & 1        & 0.5     &         &         \\
Theorem~\ref{thm:Stenger-4} & 3   & 1        & 1       & 23.5    & 1.93    \\
Theorem~\ref{thm:New1}      & 1.17& 1        & 1       & 11.3    & 1.9    \\
\hline
\end{tabular}
\caption{Parameters for a function $f_3(t)$ in~\eqref{eq:f_3}.}
\label{tbl:parameters-3}
\begin{tabular}{c|lllll}
\hline
& \multicolumn{1}{c}{$d$} & \multicolumn{1}{c}{$\alpha$} & \multicolumn{1}{c}{$\beta$} & \multicolumn{1}{c}{$K_{-}$} & \multicolumn{1}{c}{$K_{+}$} \\
\hline
Theorem~\ref{thm:Stenger-2} & 1.5 & 1 & 0.5 &     &     \\
Theorem~\ref{thm:Stenger-4} & 0.8 & 1 & 1   & 1.1 & 4.0 \\
Theorem~\ref{thm:New1}      & --  & 1 & 1 & --  & --  \\
\hline
\end{tabular}
\end{table}

\begin{figure}[htpb]
\centering
\scalebox{0.72}{\begin{tikzpicture}[gnuplot]
%% generated with GNUPLOT 5.0p7 (Lua 5.3; terminal rev. 99, script rev. 100)
%% 火  9/29 00:38:37 2026
\tikzset{every node/.append style={font={\ttfamily\fontsize{12.0pt}{14.4pt}\selectfont}}}
\gpmonochromelines
\path (0.000,0.000) rectangle (12.500,8.750);
\gpcolor{color=gp lt color border}
\gpsetlinetype{gp lt border}
\gpsetdashtype{gp dt solid}
\gpsetlinewidth{1.00}
\draw[gp path] (2.027,1.184)--(2.117,1.184);
\draw[gp path] (11.836,1.184)--(11.746,1.184);
\draw[gp path] (2.027,1.523)--(2.117,1.523);
\draw[gp path] (11.836,1.523)--(11.746,1.523);
\draw[gp path] (2.027,1.862)--(2.207,1.862);
\draw[gp path] (11.836,1.862)--(11.656,1.862);
\node[gp node right] at (1.806,1.862) {$10^{-15}$};
\draw[gp path] (2.027,2.202)--(2.117,2.202);
\draw[gp path] (11.836,2.202)--(11.746,2.202);
\draw[gp path] (2.027,2.541)--(2.117,2.541);
\draw[gp path] (11.836,2.541)--(11.746,2.541);
\draw[gp path] (2.027,2.880)--(2.117,2.880);
\draw[gp path] (11.836,2.880)--(11.746,2.880);
\draw[gp path] (2.027,3.219)--(2.117,3.219);
\draw[gp path] (11.836,3.219)--(11.746,3.219);
\draw[gp path] (2.027,3.558)--(2.207,3.558);
\draw[gp path] (11.836,3.558)--(11.656,3.558);
\node[gp node right] at (1.806,3.558) {$10^{-10}$};
\draw[gp path] (2.027,3.898)--(2.117,3.898);
\draw[gp path] (11.836,3.898)--(11.746,3.898);
\draw[gp path] (2.027,4.237)--(2.117,4.237);
\draw[gp path] (11.836,4.237)--(11.746,4.237);
\draw[gp path] (2.027,4.576)--(2.117,4.576);
\draw[gp path] (11.836,4.576)--(11.746,4.576);
\draw[gp path] (2.027,4.915)--(2.117,4.915);
\draw[gp path] (11.836,4.915)--(11.746,4.915);
\draw[gp path] (2.027,5.254)--(2.207,5.254);
\draw[gp path] (11.836,5.254)--(11.656,5.254);
\node[gp node right] at (1.806,5.254) {$10^{-5}$};
\draw[gp path] (2.027,5.593)--(2.117,5.593);
\draw[gp path] (11.836,5.593)--(11.746,5.593);
\draw[gp path] (2.027,5.933)--(2.117,5.933);
\draw[gp path] (11.836,5.933)--(11.746,5.933);
\draw[gp path] (2.027,6.272)--(2.117,6.272);
\draw[gp path] (11.836,6.272)--(11.746,6.272);
\draw[gp path] (2.027,6.611)--(2.117,6.611);
\draw[gp path] (11.836,6.611)--(11.746,6.611);
\draw[gp path] (2.027,6.950)--(2.207,6.950);
\draw[gp path] (11.836,6.950)--(11.656,6.950);
\node[gp node right] at (1.806,6.950) {$10^{0}$};
\draw[gp path] (2.027,7.289)--(2.117,7.289);
\draw[gp path] (11.836,7.289)--(11.746,7.289);
\draw[gp path] (2.027,7.629)--(2.117,7.629);
\draw[gp path] (11.836,7.629)--(11.746,7.629);
\draw[gp path] (2.027,7.968)--(2.117,7.968);
\draw[gp path] (11.836,7.968)--(11.746,7.968);
\draw[gp path] (2.027,8.307)--(2.117,8.307);
\draw[gp path] (11.836,8.307)--(11.746,8.307);
\draw[gp path] (2.027,1.184)--(2.027,1.364);
\draw[gp path] (2.027,8.307)--(2.027,8.127);
\node[gp node center] at (2.027,0.814) {$0$};
\draw[gp path] (3.428,1.184)--(3.428,1.364);
\draw[gp path] (3.428,8.307)--(3.428,8.127);
\node[gp node center] at (3.428,0.814) {$20$};
\draw[gp path] (4.830,1.184)--(4.830,1.364);
\draw[gp path] (4.830,8.307)--(4.830,8.127);
\node[gp node center] at (4.830,0.814) {$40$};
\draw[gp path] (6.231,1.184)--(6.231,1.364);
\draw[gp path] (6.231,8.307)--(6.231,8.127);
\node[gp node center] at (6.231,0.814) {$60$};
\draw[gp path] (7.632,1.184)--(7.632,1.364);
\draw[gp path] (7.632,8.307)--(7.632,8.127);
\node[gp node center] at (7.632,0.814) {$80$};
\draw[gp path] (9.033,1.184)--(9.033,1.364);
\draw[gp path] (9.033,8.307)--(9.033,8.127);
\node[gp node center] at (9.033,0.814) {$100$};
\draw[gp path] (10.435,1.184)--(10.435,1.364);
\draw[gp path] (10.435,8.307)--(10.435,8.127);
\node[gp node center] at (10.435,0.814) {$120$};
\draw[gp path] (11.836,1.184)--(11.836,1.364);
\draw[gp path] (11.836,8.307)--(11.836,8.127);
\node[gp node center] at (11.836,0.814) {$140$};
\draw[gp path] (2.027,8.307)--(2.027,1.184)--(11.836,1.184)--(11.836,8.307)--cycle;
\node[gp node center,rotate=-270] at (0.295,4.745) {maximum error};
\node[gp node center] at (6.931,0.259) {$n$};
\node[gp node right] at (10.109,7.902) {Observed error (Theorem 2.1)};
\draw[gp path] (10.330,7.902)--(11.394,7.902);
\draw[gp path] (2.167,6.195)--(2.517,5.972)--(2.868,5.886)--(3.218,5.701)--(3.568,5.540)%
  --(3.919,5.395)--(4.269,5.320)--(4.619,5.192)--(4.970,5.078)--(5.320,4.955)--(5.670,4.890)%
  --(6.021,4.781)--(6.371,4.690)--(6.721,4.609)--(7.072,4.565)--(7.422,4.474)--(7.772,4.357)%
  --(8.123,4.302)--(8.473,4.267)--(8.823,4.149)--(9.174,4.109)--(9.524,3.992)--(9.874,3.970)%
  --(10.225,3.903)--(10.575,3.851)--(10.925,3.724)--(11.275,3.742)--(11.626,3.621)--(11.836,3.606);
\gpsetpointsize{4.00}
\gppoint{gp mark 2}{(2.167,6.195)}
\gppoint{gp mark 2}{(2.517,5.972)}
\gppoint{gp mark 2}{(2.868,5.886)}
\gppoint{gp mark 2}{(3.218,5.701)}
\gppoint{gp mark 2}{(3.568,5.540)}
\gppoint{gp mark 2}{(3.919,5.395)}
\gppoint{gp mark 2}{(4.269,5.320)}
\gppoint{gp mark 2}{(4.619,5.192)}
\gppoint{gp mark 2}{(4.970,5.078)}
\gppoint{gp mark 2}{(5.320,4.955)}
\gppoint{gp mark 2}{(5.670,4.890)}
\gppoint{gp mark 2}{(6.021,4.781)}
\gppoint{gp mark 2}{(6.371,4.690)}
\gppoint{gp mark 2}{(6.721,4.609)}
\gppoint{gp mark 2}{(7.072,4.565)}
\gppoint{gp mark 2}{(7.422,4.474)}
\gppoint{gp mark 2}{(7.772,4.357)}
\gppoint{gp mark 2}{(8.123,4.302)}
\gppoint{gp mark 2}{(8.473,4.267)}
\gppoint{gp mark 2}{(8.823,4.149)}
\gppoint{gp mark 2}{(9.174,4.109)}
\gppoint{gp mark 2}{(9.524,3.992)}
\gppoint{gp mark 2}{(9.874,3.970)}
\gppoint{gp mark 2}{(10.225,3.903)}
\gppoint{gp mark 2}{(10.575,3.851)}
\gppoint{gp mark 2}{(10.925,3.724)}
\gppoint{gp mark 2}{(11.275,3.742)}
\gppoint{gp mark 2}{(11.626,3.621)}
\gppoint{gp mark 2}{(10.862,7.902)}
\node[gp node right] at (10.109,7.452) {Observed error (Theorem 2.2)};
\draw[gp path] (10.330,7.452)--(11.394,7.452);
\draw[gp path] (2.167,6.266)--(2.517,5.800)--(2.868,5.467)--(3.218,5.194)--(3.568,4.951)%
  --(3.919,4.726)--(4.269,4.532)--(4.619,4.347)--(4.970,4.166)--(5.320,4.008)--(5.670,3.844)%
  --(6.021,3.698)--(6.371,3.553)--(6.721,3.414)--(7.072,3.274)--(7.422,3.143)--(7.772,3.022)%
  --(8.123,2.902)--(8.473,2.784)--(8.823,2.669)--(9.174,2.549)--(9.524,2.441)--(9.874,2.330)%
  --(10.225,2.219)--(10.575,2.123)--(10.925,2.024)--(11.275,1.921)--(11.626,1.820)--(11.836,1.765);
\gppoint{gp mark 9}{(2.167,6.266)}
\gppoint{gp mark 9}{(2.517,5.800)}
\gppoint{gp mark 9}{(2.868,5.467)}
\gppoint{gp mark 9}{(3.218,5.194)}
\gppoint{gp mark 9}{(3.568,4.951)}
\gppoint{gp mark 9}{(3.919,4.726)}
\gppoint{gp mark 9}{(4.269,4.532)}
\gppoint{gp mark 9}{(4.619,4.347)}
\gppoint{gp mark 9}{(4.970,4.166)}
\gppoint{gp mark 9}{(5.320,4.008)}
\gppoint{gp mark 9}{(5.670,3.844)}
\gppoint{gp mark 9}{(6.021,3.698)}
\gppoint{gp mark 9}{(6.371,3.553)}
\gppoint{gp mark 9}{(6.721,3.414)}
\gppoint{gp mark 9}{(7.072,3.274)}
\gppoint{gp mark 9}{(7.422,3.143)}
\gppoint{gp mark 9}{(7.772,3.022)}
\gppoint{gp mark 9}{(8.123,2.902)}
\gppoint{gp mark 9}{(8.473,2.784)}
\gppoint{gp mark 9}{(8.823,2.669)}
\gppoint{gp mark 9}{(9.174,2.549)}
\gppoint{gp mark 9}{(9.524,2.441)}
\gppoint{gp mark 9}{(9.874,2.330)}
\gppoint{gp mark 9}{(10.225,2.219)}
\gppoint{gp mark 9}{(10.575,2.123)}
\gppoint{gp mark 9}{(10.925,2.024)}
\gppoint{gp mark 9}{(11.275,1.921)}
\gppoint{gp mark 9}{(11.626,1.820)}
\gppoint{gp mark 9}{(10.862,7.452)}
\node[gp node right] at (10.109,7.002) {Error bound (Theorem 2.2)};
\gpsetdashtype{gp dt 3}
\draw[gp path] (10.330,7.002)--(11.394,7.002);
\draw[gp path] (2.167,7.559)--(2.517,7.094)--(2.868,6.764)--(3.218,6.491)--(3.568,6.254)%
  --(3.919,6.040)--(4.269,5.844)--(4.619,5.662)--(4.970,5.492)--(5.320,5.330)--(5.670,5.177)%
  --(6.021,5.031)--(6.371,4.890)--(6.721,4.755)--(7.072,4.625)--(7.422,4.499)--(7.772,4.377)%
  --(8.123,4.258)--(8.473,4.143)--(8.823,4.030)--(9.174,3.920)--(9.524,3.813)--(9.874,3.709)%
  --(10.225,3.606)--(10.575,3.506)--(10.925,3.408)--(11.275,3.311)--(11.626,3.216)--(11.836,3.160);
\node[gp node right] at (10.109,6.552) {Observed error (Theorem 2.3)};
\gpsetdashtype{gp dt solid}
\draw[gp path] (10.330,6.552)--(11.394,6.552);
\draw[gp path] (2.167,6.319)--(2.517,5.764)--(2.868,5.175)--(3.218,4.576)--(3.568,3.973)%
  --(3.919,3.392)--(4.269,2.805)--(4.619,2.234)--(4.970,1.666)--(5.320,1.394)--(5.670,1.394)%
  --(6.021,1.436);
\gppoint{gp mark 7}{(2.167,6.319)}
\gppoint{gp mark 7}{(2.517,5.764)}
\gppoint{gp mark 7}{(2.868,5.175)}
\gppoint{gp mark 7}{(3.218,4.576)}
\gppoint{gp mark 7}{(3.568,3.973)}
\gppoint{gp mark 7}{(3.919,3.392)}
\gppoint{gp mark 7}{(4.269,2.805)}
\gppoint{gp mark 7}{(4.619,2.234)}
\gppoint{gp mark 7}{(4.970,1.666)}
\gppoint{gp mark 7}{(5.320,1.394)}
\gppoint{gp mark 7}{(5.670,1.394)}
\gppoint{gp mark 7}{(6.021,1.436)}
\gppoint{gp mark 7}{(10.862,6.552)}
\node[gp node right] at (10.109,6.102) {Error bound (Theorem 2.3)};
\gpsetdashtype{gp dt 5}
\gpsetlinewidth{2.00}
\draw[gp path] (10.330,6.102)--(11.394,6.102);
\draw[gp path] (2.167,7.716)--(2.517,7.062)--(2.868,6.470)--(3.218,5.919)--(3.568,5.395)%
  --(3.919,4.892)--(4.269,4.403)--(4.619,3.927)--(4.970,3.461)--(5.320,3.004)--(5.670,2.554)%
  --(6.021,2.110);
\gpsetdashtype{gp dt solid}
\gpsetlinewidth{1.00}
\draw[gp path] (2.027,8.307)--(2.027,1.184)--(11.836,1.184)--(11.836,8.307)--cycle;
%% coordinates of the plot area
\gpdefrectangularnode{gp plot 1}{\pgfpoint{2.027cm}{1.184cm}}{\pgfpoint{11.836cm}{8.307cm}}
\end{tikzpicture}
%% gnuplot variables
}
\caption{Observed errors and error bounds for $f_1(t)$ in~\eqref{eq:f_1}.}\label{fig:example1}
\scalebox{0.72}{\begin{tikzpicture}[gnuplot]
%% generated with GNUPLOT 5.0p7 (Lua 5.3; terminal rev. 99, script rev. 100)
%% 火  9/29 00:38:51 2026
\tikzset{every node/.append style={font={\ttfamily\fontsize{12.0pt}{14.4pt}\selectfont}}}
\gpmonochromelines
\path (0.000,0.000) rectangle (12.500,8.750);
\gpcolor{color=gp lt color border}
\gpsetlinetype{gp lt border}
\gpsetdashtype{gp dt solid}
\gpsetlinewidth{1.00}
\draw[gp path] (2.027,1.184)--(2.117,1.184);
\draw[gp path] (11.836,1.184)--(11.746,1.184);
\draw[gp path] (2.027,1.523)--(2.117,1.523);
\draw[gp path] (11.836,1.523)--(11.746,1.523);
\draw[gp path] (2.027,1.862)--(2.207,1.862);
\draw[gp path] (11.836,1.862)--(11.656,1.862);
\node[gp node right] at (1.806,1.862) {$10^{-15}$};
\draw[gp path] (2.027,2.202)--(2.117,2.202);
\draw[gp path] (11.836,2.202)--(11.746,2.202);
\draw[gp path] (2.027,2.541)--(2.117,2.541);
\draw[gp path] (11.836,2.541)--(11.746,2.541);
\draw[gp path] (2.027,2.880)--(2.117,2.880);
\draw[gp path] (11.836,2.880)--(11.746,2.880);
\draw[gp path] (2.027,3.219)--(2.117,3.219);
\draw[gp path] (11.836,3.219)--(11.746,3.219);
\draw[gp path] (2.027,3.558)--(2.207,3.558);
\draw[gp path] (11.836,3.558)--(11.656,3.558);
\node[gp node right] at (1.806,3.558) {$10^{-10}$};
\draw[gp path] (2.027,3.898)--(2.117,3.898);
\draw[gp path] (11.836,3.898)--(11.746,3.898);
\draw[gp path] (2.027,4.237)--(2.117,4.237);
\draw[gp path] (11.836,4.237)--(11.746,4.237);
\draw[gp path] (2.027,4.576)--(2.117,4.576);
\draw[gp path] (11.836,4.576)--(11.746,4.576);
\draw[gp path] (2.027,4.915)--(2.117,4.915);
\draw[gp path] (11.836,4.915)--(11.746,4.915);
\draw[gp path] (2.027,5.254)--(2.207,5.254);
\draw[gp path] (11.836,5.254)--(11.656,5.254);
\node[gp node right] at (1.806,5.254) {$10^{-5}$};
\draw[gp path] (2.027,5.593)--(2.117,5.593);
\draw[gp path] (11.836,5.593)--(11.746,5.593);
\draw[gp path] (2.027,5.933)--(2.117,5.933);
\draw[gp path] (11.836,5.933)--(11.746,5.933);
\draw[gp path] (2.027,6.272)--(2.117,6.272);
\draw[gp path] (11.836,6.272)--(11.746,6.272);
\draw[gp path] (2.027,6.611)--(2.117,6.611);
\draw[gp path] (11.836,6.611)--(11.746,6.611);
\draw[gp path] (2.027,6.950)--(2.207,6.950);
\draw[gp path] (11.836,6.950)--(11.656,6.950);
\node[gp node right] at (1.806,6.950) {$10^{0}$};
\draw[gp path] (2.027,7.289)--(2.117,7.289);
\draw[gp path] (11.836,7.289)--(11.746,7.289);
\draw[gp path] (2.027,7.629)--(2.117,7.629);
\draw[gp path] (11.836,7.629)--(11.746,7.629);
\draw[gp path] (2.027,7.968)--(2.117,7.968);
\draw[gp path] (11.836,7.968)--(11.746,7.968);
\draw[gp path] (2.027,8.307)--(2.117,8.307);
\draw[gp path] (11.836,8.307)--(11.746,8.307);
\draw[gp path] (2.027,1.184)--(2.027,1.364);
\draw[gp path] (2.027,8.307)--(2.027,8.127);
\node[gp node center] at (2.027,0.814) {$0$};
\draw[gp path] (3.428,1.184)--(3.428,1.364);
\draw[gp path] (3.428,8.307)--(3.428,8.127);
\node[gp node center] at (3.428,0.814) {$20$};
\draw[gp path] (4.830,1.184)--(4.830,1.364);
\draw[gp path] (4.830,8.307)--(4.830,8.127);
\node[gp node center] at (4.830,0.814) {$40$};
\draw[gp path] (6.231,1.184)--(6.231,1.364);
\draw[gp path] (6.231,8.307)--(6.231,8.127);
\node[gp node center] at (6.231,0.814) {$60$};
\draw[gp path] (7.632,1.184)--(7.632,1.364);
\draw[gp path] (7.632,8.307)--(7.632,8.127);
\node[gp node center] at (7.632,0.814) {$80$};
\draw[gp path] (9.033,1.184)--(9.033,1.364);
\draw[gp path] (9.033,8.307)--(9.033,8.127);
\node[gp node center] at (9.033,0.814) {$100$};
\draw[gp path] (10.435,1.184)--(10.435,1.364);
\draw[gp path] (10.435,8.307)--(10.435,8.127);
\node[gp node center] at (10.435,0.814) {$120$};
\draw[gp path] (11.836,1.184)--(11.836,1.364);
\draw[gp path] (11.836,8.307)--(11.836,8.127);
\node[gp node center] at (11.836,0.814) {$140$};
\draw[gp path] (2.027,8.307)--(2.027,1.184)--(11.836,1.184)--(11.836,8.307)--cycle;
\node[gp node center,rotate=-270] at (0.295,4.745) {maximum error};
\node[gp node center] at (6.931,0.259) {$n$};
\node[gp node right] at (10.109,7.902) {Observed error (Theorem 2.1)};
\draw[gp path] (10.330,7.902)--(11.394,7.902);
\draw[gp path] (2.167,6.417)--(2.517,6.196)--(2.868,6.139)--(3.218,5.957)--(3.568,5.900)%
  --(3.919,5.752)--(4.269,5.699)--(4.619,5.558)--(4.970,5.513)--(5.320,5.396)--(5.670,5.346)%
  --(6.021,5.238)--(6.371,5.215)--(6.721,5.125)--(7.072,5.088)--(7.422,4.983)--(7.772,4.947)%
  --(8.123,4.880)--(8.473,4.829)--(8.823,4.759)--(9.174,4.730)--(9.524,4.636)--(9.874,4.627)%
  --(10.225,4.525)--(10.575,4.524)--(10.925,4.432)--(11.275,4.423)--(11.626,4.350)--(11.836,4.330);
\gpsetpointsize{4.00}
\gppoint{gp mark 2}{(2.167,6.417)}
\gppoint{gp mark 2}{(2.517,6.196)}
\gppoint{gp mark 2}{(2.868,6.139)}
\gppoint{gp mark 2}{(3.218,5.957)}
\gppoint{gp mark 2}{(3.568,5.900)}
\gppoint{gp mark 2}{(3.919,5.752)}
\gppoint{gp mark 2}{(4.269,5.699)}
\gppoint{gp mark 2}{(4.619,5.558)}
\gppoint{gp mark 2}{(4.970,5.513)}
\gppoint{gp mark 2}{(5.320,5.396)}
\gppoint{gp mark 2}{(5.670,5.346)}
\gppoint{gp mark 2}{(6.021,5.238)}
\gppoint{gp mark 2}{(6.371,5.215)}
\gppoint{gp mark 2}{(6.721,5.125)}
\gppoint{gp mark 2}{(7.072,5.088)}
\gppoint{gp mark 2}{(7.422,4.983)}
\gppoint{gp mark 2}{(7.772,4.947)}
\gppoint{gp mark 2}{(8.123,4.880)}
\gppoint{gp mark 2}{(8.473,4.829)}
\gppoint{gp mark 2}{(8.823,4.759)}
\gppoint{gp mark 2}{(9.174,4.730)}
\gppoint{gp mark 2}{(9.524,4.636)}
\gppoint{gp mark 2}{(9.874,4.627)}
\gppoint{gp mark 2}{(10.225,4.525)}
\gppoint{gp mark 2}{(10.575,4.524)}
\gppoint{gp mark 2}{(10.925,4.432)}
\gppoint{gp mark 2}{(11.275,4.423)}
\gppoint{gp mark 2}{(11.626,4.350)}
\gppoint{gp mark 2}{(10.862,7.902)}
\node[gp node right] at (10.109,7.452) {Observed error (Theorem 2.2)};
\draw[gp path] (10.330,7.452)--(11.394,7.452);
\draw[gp path] (2.167,6.276)--(2.517,5.794)--(2.868,5.442)--(3.218,5.157)--(3.568,4.903)%
  --(3.919,4.673)--(4.269,4.462)--(4.619,4.261)--(4.970,4.069)--(5.320,3.893)--(5.670,3.720)%
  --(6.021,3.552)--(6.371,3.399)--(6.721,3.242)--(7.072,3.084)--(7.422,2.948)--(7.772,2.790)%
  --(8.123,2.699)--(8.473,2.569)--(8.823,2.427)--(9.174,2.314)--(9.524,2.234)--(9.874,2.118)%
  --(10.225,2.030)--(10.575,1.916)--(10.925,1.598)--(11.275,1.539)--(11.626,1.539)--(11.836,1.513);
\gppoint{gp mark 9}{(2.167,6.276)}
\gppoint{gp mark 9}{(2.517,5.794)}
\gppoint{gp mark 9}{(2.868,5.442)}
\gppoint{gp mark 9}{(3.218,5.157)}
\gppoint{gp mark 9}{(3.568,4.903)}
\gppoint{gp mark 9}{(3.919,4.673)}
\gppoint{gp mark 9}{(4.269,4.462)}
\gppoint{gp mark 9}{(4.619,4.261)}
\gppoint{gp mark 9}{(4.970,4.069)}
\gppoint{gp mark 9}{(5.320,3.893)}
\gppoint{gp mark 9}{(5.670,3.720)}
\gppoint{gp mark 9}{(6.021,3.552)}
\gppoint{gp mark 9}{(6.371,3.399)}
\gppoint{gp mark 9}{(6.721,3.242)}
\gppoint{gp mark 9}{(7.072,3.084)}
\gppoint{gp mark 9}{(7.422,2.948)}
\gppoint{gp mark 9}{(7.772,2.790)}
\gppoint{gp mark 9}{(8.123,2.699)}
\gppoint{gp mark 9}{(8.473,2.569)}
\gppoint{gp mark 9}{(8.823,2.427)}
\gppoint{gp mark 9}{(9.174,2.314)}
\gppoint{gp mark 9}{(9.524,2.234)}
\gppoint{gp mark 9}{(9.874,2.118)}
\gppoint{gp mark 9}{(10.225,2.030)}
\gppoint{gp mark 9}{(10.575,1.916)}
\gppoint{gp mark 9}{(10.925,1.598)}
\gppoint{gp mark 9}{(11.275,1.539)}
\gppoint{gp mark 9}{(11.626,1.539)}
\gppoint{gp mark 9}{(10.862,7.452)}
\node[gp node right] at (10.109,7.002) {Error bound (Theorem 2.2)};
\gpsetdashtype{gp dt 3}
\draw[gp path] (10.330,7.002)--(11.394,7.002);
\draw[gp path] (2.167,7.269)--(2.517,6.804)--(2.868,6.473)--(3.218,6.201)--(3.568,5.964)%
  --(3.919,5.750)--(4.269,5.554)--(4.619,5.372)--(4.970,5.202)--(5.320,5.040)--(5.670,4.887)%
  --(6.021,4.741)--(6.371,4.600)--(6.721,4.465)--(7.072,4.335)--(7.422,4.209)--(7.772,4.086)%
  --(8.123,3.968)--(8.473,3.852)--(8.823,3.740)--(9.174,3.630)--(9.524,3.523)--(9.874,3.419)%
  --(10.225,3.316)--(10.575,3.216)--(10.925,3.117)--(11.275,3.021)--(11.626,2.926)--(11.836,2.870);
\node[gp node right] at (10.109,6.552) {Observed error (Theorem 2.3)};
\gpsetdashtype{gp dt solid}
\draw[gp path] (10.330,6.552)--(11.394,6.552);
\draw[gp path] (2.167,6.324)--(2.517,5.752)--(2.868,5.134)--(3.218,4.475)--(3.568,4.020)%
  --(3.919,3.487)--(4.269,2.966)--(4.619,2.534)--(4.970,2.020)--(5.320,1.598)--(5.670,1.539)%
  --(6.021,1.496);
\gppoint{gp mark 7}{(2.167,6.324)}
\gppoint{gp mark 7}{(2.517,5.752)}
\gppoint{gp mark 7}{(2.868,5.134)}
\gppoint{gp mark 7}{(3.218,4.475)}
\gppoint{gp mark 7}{(3.568,4.020)}
\gppoint{gp mark 7}{(3.919,3.487)}
\gppoint{gp mark 7}{(4.269,2.966)}
\gppoint{gp mark 7}{(4.619,2.534)}
\gppoint{gp mark 7}{(4.970,2.020)}
\gppoint{gp mark 7}{(5.320,1.598)}
\gppoint{gp mark 7}{(5.670,1.539)}
\gppoint{gp mark 7}{(6.021,1.496)}
\gppoint{gp mark 7}{(10.862,6.552)}
\node[gp node right] at (10.109,6.102) {Error bound (Theorem 2.3)};
\gpsetdashtype{gp dt 5}
\gpsetlinewidth{2.00}
\draw[gp path] (10.330,6.102)--(11.394,6.102);
\draw[gp path] (2.167,7.553)--(2.517,6.899)--(2.868,6.306)--(3.218,5.755)--(3.568,5.232)%
  --(3.919,4.728)--(4.269,4.240)--(4.619,3.764)--(4.970,3.298)--(5.320,2.840)--(5.670,2.390)%
  --(6.021,1.947);
\gpsetdashtype{gp dt solid}
\gpsetlinewidth{1.00}
\draw[gp path] (2.027,8.307)--(2.027,1.184)--(11.836,1.184)--(11.836,8.307)--cycle;
%% coordinates of the plot area
\gpdefrectangularnode{gp plot 1}{\pgfpoint{2.027cm}{1.184cm}}{\pgfpoint{11.836cm}{8.307cm}}
\end{tikzpicture}
%% gnuplot variables
}
\caption{Observed errors and error bounds for $f_2(t)$ in~\eqref{eq:f_2}.}\label{fig:example2}
\scalebox{0.72}{\begin{tikzpicture}[gnuplot]
%% generated with GNUPLOT 5.0p7 (Lua 5.3; terminal rev. 99, script rev. 100)
%% 月  7/27 00:33:59 2026
\tikzset{every node/.append style={font={\ttfamily\fontsize{12.0pt}{14.4pt}\selectfont}}}
\gpmonochromelines
\path (0.000,0.000) rectangle (12.500,8.750);
\gpcolor{color=gp lt color border}
\gpsetlinetype{gp lt border}
\gpsetdashtype{gp dt solid}
\gpsetlinewidth{1.00}
\draw[gp path] (2.027,1.184)--(2.117,1.184);
\draw[gp path] (11.836,1.184)--(11.746,1.184);
\draw[gp path] (2.027,1.523)--(2.117,1.523);
\draw[gp path] (11.836,1.523)--(11.746,1.523);
\draw[gp path] (2.027,1.862)--(2.207,1.862);
\draw[gp path] (11.836,1.862)--(11.656,1.862);
\node[gp node right] at (1.806,1.862) {$10^{-15}$};
\draw[gp path] (2.027,2.202)--(2.117,2.202);
\draw[gp path] (11.836,2.202)--(11.746,2.202);
\draw[gp path] (2.027,2.541)--(2.117,2.541);
\draw[gp path] (11.836,2.541)--(11.746,2.541);
\draw[gp path] (2.027,2.880)--(2.117,2.880);
\draw[gp path] (11.836,2.880)--(11.746,2.880);
\draw[gp path] (2.027,3.219)--(2.117,3.219);
\draw[gp path] (11.836,3.219)--(11.746,3.219);
\draw[gp path] (2.027,3.558)--(2.207,3.558);
\draw[gp path] (11.836,3.558)--(11.656,3.558);
\node[gp node right] at (1.806,3.558) {$10^{-10}$};
\draw[gp path] (2.027,3.898)--(2.117,3.898);
\draw[gp path] (11.836,3.898)--(11.746,3.898);
\draw[gp path] (2.027,4.237)--(2.117,4.237);
\draw[gp path] (11.836,4.237)--(11.746,4.237);
\draw[gp path] (2.027,4.576)--(2.117,4.576);
\draw[gp path] (11.836,4.576)--(11.746,4.576);
\draw[gp path] (2.027,4.915)--(2.117,4.915);
\draw[gp path] (11.836,4.915)--(11.746,4.915);
\draw[gp path] (2.027,5.254)--(2.207,5.254);
\draw[gp path] (11.836,5.254)--(11.656,5.254);
\node[gp node right] at (1.806,5.254) {$10^{-5}$};
\draw[gp path] (2.027,5.593)--(2.117,5.593);
\draw[gp path] (11.836,5.593)--(11.746,5.593);
\draw[gp path] (2.027,5.933)--(2.117,5.933);
\draw[gp path] (11.836,5.933)--(11.746,5.933);
\draw[gp path] (2.027,6.272)--(2.117,6.272);
\draw[gp path] (11.836,6.272)--(11.746,6.272);
\draw[gp path] (2.027,6.611)--(2.117,6.611);
\draw[gp path] (11.836,6.611)--(11.746,6.611);
\draw[gp path] (2.027,6.950)--(2.207,6.950);
\draw[gp path] (11.836,6.950)--(11.656,6.950);
\node[gp node right] at (1.806,6.950) {$10^{0}$};
\draw[gp path] (2.027,7.289)--(2.117,7.289);
\draw[gp path] (11.836,7.289)--(11.746,7.289);
\draw[gp path] (2.027,7.629)--(2.117,7.629);
\draw[gp path] (11.836,7.629)--(11.746,7.629);
\draw[gp path] (2.027,7.968)--(2.117,7.968);
\draw[gp path] (11.836,7.968)--(11.746,7.968);
\draw[gp path] (2.027,8.307)--(2.117,8.307);
\draw[gp path] (11.836,8.307)--(11.746,8.307);
\draw[gp path] (2.027,1.184)--(2.027,1.364);
\draw[gp path] (2.027,8.307)--(2.027,8.127);
\node[gp node center] at (2.027,0.814) {$0$};
\draw[gp path] (3.428,1.184)--(3.428,1.364);
\draw[gp path] (3.428,8.307)--(3.428,8.127);
\node[gp node center] at (3.428,0.814) {$20$};
\draw[gp path] (4.830,1.184)--(4.830,1.364);
\draw[gp path] (4.830,8.307)--(4.830,8.127);
\node[gp node center] at (4.830,0.814) {$40$};
\draw[gp path] (6.231,1.184)--(6.231,1.364);
\draw[gp path] (6.231,8.307)--(6.231,8.127);
\node[gp node center] at (6.231,0.814) {$60$};
\draw[gp path] (7.632,1.184)--(7.632,1.364);
\draw[gp path] (7.632,8.307)--(7.632,8.127);
\node[gp node center] at (7.632,0.814) {$80$};
\draw[gp path] (9.033,1.184)--(9.033,1.364);
\draw[gp path] (9.033,8.307)--(9.033,8.127);
\node[gp node center] at (9.033,0.814) {$100$};
\draw[gp path] (10.435,1.184)--(10.435,1.364);
\draw[gp path] (10.435,8.307)--(10.435,8.127);
\node[gp node center] at (10.435,0.814) {$120$};
\draw[gp path] (11.836,1.184)--(11.836,1.364);
\draw[gp path] (11.836,8.307)--(11.836,8.127);
\node[gp node center] at (11.836,0.814) {$140$};
\draw[gp path] (2.027,8.307)--(2.027,1.184)--(11.836,1.184)--(11.836,8.307)--cycle;
\node[gp node center,rotate=-270] at (0.295,4.745) {maximum error};
\node[gp node center] at (6.931,0.259) {$n$};
\node[gp node right] at (10.109,7.902) {Observed error (Theorem 2.1)};
\draw[gp path] (10.330,7.902)--(11.394,7.902);
\draw[gp path] (2.167,6.566)--(2.517,6.254)--(2.868,6.142)--(3.218,5.957)--(3.568,5.898)%
  --(3.919,5.750)--(4.269,5.696)--(4.619,5.557)--(4.970,5.512)--(5.320,5.396)--(5.670,5.346)%
  --(6.021,5.238)--(6.371,5.215)--(6.721,5.124)--(7.072,5.085)--(7.422,4.988)--(7.772,4.945)%
  --(8.123,4.876)--(8.473,4.824)--(8.823,4.759)--(9.174,4.733)--(9.524,4.637)--(9.874,4.625)%
  --(10.225,4.524)--(10.575,4.526)--(10.925,4.433)--(11.275,4.421)--(11.626,4.351)--(11.836,4.332);
\gpsetpointsize{4.00}
\gppoint{gp mark 2}{(2.167,6.566)}
\gppoint{gp mark 2}{(2.517,6.254)}
\gppoint{gp mark 2}{(2.868,6.142)}
\gppoint{gp mark 2}{(3.218,5.957)}
\gppoint{gp mark 2}{(3.568,5.898)}
\gppoint{gp mark 2}{(3.919,5.750)}
\gppoint{gp mark 2}{(4.269,5.696)}
\gppoint{gp mark 2}{(4.619,5.557)}
\gppoint{gp mark 2}{(4.970,5.512)}
\gppoint{gp mark 2}{(5.320,5.396)}
\gppoint{gp mark 2}{(5.670,5.346)}
\gppoint{gp mark 2}{(6.021,5.238)}
\gppoint{gp mark 2}{(6.371,5.215)}
\gppoint{gp mark 2}{(6.721,5.124)}
\gppoint{gp mark 2}{(7.072,5.085)}
\gppoint{gp mark 2}{(7.422,4.988)}
\gppoint{gp mark 2}{(7.772,4.945)}
\gppoint{gp mark 2}{(8.123,4.876)}
\gppoint{gp mark 2}{(8.473,4.824)}
\gppoint{gp mark 2}{(8.823,4.759)}
\gppoint{gp mark 2}{(9.174,4.733)}
\gppoint{gp mark 2}{(9.524,4.637)}
\gppoint{gp mark 2}{(9.874,4.625)}
\gppoint{gp mark 2}{(10.225,4.524)}
\gppoint{gp mark 2}{(10.575,4.526)}
\gppoint{gp mark 2}{(10.925,4.433)}
\gppoint{gp mark 2}{(11.275,4.421)}
\gppoint{gp mark 2}{(11.626,4.351)}
\gppoint{gp mark 2}{(10.862,7.902)}
\node[gp node right] at (10.109,7.452) {Observed error (Theorem 2.2)};
\draw[gp path] (10.330,7.452)--(11.394,7.452);
\draw[gp path] (2.167,6.480)--(2.517,6.249)--(2.868,6.065)--(3.218,5.933)--(3.568,5.796)%
  --(3.919,5.702)--(4.269,5.597)--(4.619,5.492)--(4.970,5.374)--(5.320,5.312)--(5.670,5.239)%
  --(6.021,5.128)--(6.371,5.081)--(6.721,5.010)--(7.072,4.896)--(7.422,4.875)--(7.772,4.820)%
  --(8.123,4.699)--(8.473,4.638)--(8.823,4.599)--(9.174,4.563)--(9.524,4.522)--(9.874,4.470)%
  --(10.225,4.417)--(10.575,4.359)--(10.925,4.306)--(11.275,4.255)--(11.626,4.209)--(11.836,4.180);
\gppoint{gp mark 9}{(2.167,6.480)}
\gppoint{gp mark 9}{(2.517,6.249)}
\gppoint{gp mark 9}{(2.868,6.065)}
\gppoint{gp mark 9}{(3.218,5.933)}
\gppoint{gp mark 9}{(3.568,5.796)}
\gppoint{gp mark 9}{(3.919,5.702)}
\gppoint{gp mark 9}{(4.269,5.597)}
\gppoint{gp mark 9}{(4.619,5.492)}
\gppoint{gp mark 9}{(4.970,5.374)}
\gppoint{gp mark 9}{(5.320,5.312)}
\gppoint{gp mark 9}{(5.670,5.239)}
\gppoint{gp mark 9}{(6.021,5.128)}
\gppoint{gp mark 9}{(6.371,5.081)}
\gppoint{gp mark 9}{(6.721,5.010)}
\gppoint{gp mark 9}{(7.072,4.896)}
\gppoint{gp mark 9}{(7.422,4.875)}
\gppoint{gp mark 9}{(7.772,4.820)}
\gppoint{gp mark 9}{(8.123,4.699)}
\gppoint{gp mark 9}{(8.473,4.638)}
\gppoint{gp mark 9}{(8.823,4.599)}
\gppoint{gp mark 9}{(9.174,4.563)}
\gppoint{gp mark 9}{(9.524,4.522)}
\gppoint{gp mark 9}{(9.874,4.470)}
\gppoint{gp mark 9}{(10.225,4.417)}
\gppoint{gp mark 9}{(10.575,4.359)}
\gppoint{gp mark 9}{(10.925,4.306)}
\gppoint{gp mark 9}{(11.275,4.255)}
\gppoint{gp mark 9}{(11.626,4.209)}
\gppoint{gp mark 9}{(10.862,7.452)}
\node[gp node right] at (10.109,7.002) {Error bound (Theorem 2.2)};
\gpsetdashtype{gp dt 3}
\draw[gp path] (10.330,7.002)--(11.394,7.002);
\draw[gp path] (2.167,7.176)--(2.517,6.980)--(2.868,6.829)--(3.218,6.701)--(3.568,6.587)%
  --(3.919,6.484)--(4.269,6.389)--(4.619,6.300)--(4.970,6.217)--(5.320,6.137)--(5.670,6.062)%
  --(6.021,5.989)--(6.371,5.920)--(6.721,5.853)--(7.072,5.788)--(7.422,5.726)--(7.772,5.665)%
  --(8.123,5.606)--(8.473,5.548)--(8.823,5.492)--(9.174,5.437)--(9.524,5.383)--(9.874,5.331)%
  --(10.225,5.280)--(10.575,5.229)--(10.925,5.180)--(11.275,5.131)--(11.626,5.084)--(11.836,5.056);
\node[gp node right] at (10.109,6.552) {Observed error (Theorem 2.3)};
\gpsetdashtype{gp dt solid}
\draw[gp path] (10.330,6.552)--(11.394,6.552);
\draw[gp path] (2.167,6.643)--(2.517,6.487)--(2.868,6.409)--(3.218,6.276)--(3.568,6.139)%
  --(3.919,5.983)--(4.269,5.818)--(4.619,5.735)--(4.970,5.665)--(5.320,5.578)--(5.670,5.484)%
  --(6.021,5.378)--(6.371,5.260)--(6.721,5.135)--(7.072,4.981)--(7.422,4.876)--(7.772,4.824)%
  --(8.123,4.776)--(8.473,4.699)--(8.823,4.609)--(9.174,4.510)--(9.524,4.400)--(9.874,4.250)%
  --(10.225,4.145)--(10.575,4.066)--(10.925,4.005)--(11.275,3.956)--(11.626,3.881);
\gppoint{gp mark 7}{(2.167,6.643)}
\gppoint{gp mark 7}{(2.517,6.487)}
\gppoint{gp mark 7}{(2.868,6.409)}
\gppoint{gp mark 7}{(3.218,6.276)}
\gppoint{gp mark 7}{(3.568,6.139)}
\gppoint{gp mark 7}{(3.919,5.983)}
\gppoint{gp mark 7}{(4.269,5.818)}
\gppoint{gp mark 7}{(4.619,5.735)}
\gppoint{gp mark 7}{(4.970,5.665)}
\gppoint{gp mark 7}{(5.320,5.578)}
\gppoint{gp mark 7}{(5.670,5.484)}
\gppoint{gp mark 7}{(6.021,5.378)}
\gppoint{gp mark 7}{(6.371,5.260)}
\gppoint{gp mark 7}{(6.721,5.135)}
\gppoint{gp mark 7}{(7.072,4.981)}
\gppoint{gp mark 7}{(7.422,4.876)}
\gppoint{gp mark 7}{(7.772,4.824)}
\gppoint{gp mark 7}{(8.123,4.776)}
\gppoint{gp mark 7}{(8.473,4.699)}
\gppoint{gp mark 7}{(8.823,4.609)}
\gppoint{gp mark 7}{(9.174,4.510)}
\gppoint{gp mark 7}{(9.524,4.400)}
\gppoint{gp mark 7}{(9.874,4.250)}
\gppoint{gp mark 7}{(10.225,4.145)}
\gppoint{gp mark 7}{(10.575,4.066)}
\gppoint{gp mark 7}{(10.925,4.005)}
\gppoint{gp mark 7}{(11.275,3.956)}
\gppoint{gp mark 7}{(11.626,3.881)}
\gppoint{gp mark 7}{(10.862,6.552)}
\draw[gp path] (2.027,8.307)--(2.027,1.184)--(11.836,1.184)--(11.836,8.307)--cycle;
%% coordinates of the plot area
\gpdefrectangularnode{gp plot 1}{\pgfpoint{2.027cm}{1.184cm}}{\pgfpoint{11.836cm}{8.307cm}}
\end{tikzpicture}
%% gnuplot variables
}
\caption{Observed errors and the error bound from
Theorem~\ref{thm:Stenger-4} for $f_3(t)$ in~\eqref{eq:f_3}.
The DE-Sinc curve shows only the observed error obtained using the
exploratory parameters $\alpha=\beta=1$ and $d=1.17$.
}\label{fig:example3}
\end{figure}

\section{Proofs}
\label{sec:proofs}

This section presents the proof of Theorem~\ref{thm:New1}.
It is organized as follows.
In Section~\ref{sec:sketch-proof},
the proof is decomposed into two lemmas:
Lemmas~\ref{lem:bound-None} and~\ref{lem:truncation-error}.
To establish these lemmas,
useful inequalities are presented in
Sections~\ref{subsec:domDdplus}, \ref{subsec:domDdminus},
and~\ref{subsec:domDd}.
Subsequently,
Lemma~\ref{lem:bound-None} is proved in Section~\ref{subsec:discretization-error},
and Lemma~\ref{lem:truncation-error} is proved in Section~\ref{subsec:truncation-error}.

Throughout this section, the two logarithms appearing in $\phi_5$
are understood as analytic branches on $\domD_{\pi/2}$ obtained by
analytic continuation from the real axis.
More precisely, $\log(1+\E^{\pi\sinh\zeta})$ is normalized to be
real for real $\zeta$, and the outer logarithm is defined similarly.
These branches are well defined because
$1+\E^{\pi\sinh\zeta}\neq 0$ and
$\log(1+\E^{\pi\sinh\zeta})\neq 0$ in $\domD_{\pi/2}$.

\subsection{Sketch of the proof}
\label{sec:sketch-proof}

Define $F(\zeta)=f(\phi_5(\zeta))$ for $\zeta\in\domD_d$.
Since $\phi_5$ maps $\mathbb{R}$ bijectively onto $\mathbb{R}$,
for each $t\in\mathbb{R}$, put $x=\phi_5^{-1}(t)$.
The main strategy in the proof of Theorem~\ref{thm:New1} is
to split the error into two terms as follows:
\begin{align*}
&\left|
f(t) -
\sum_{k=-M}^N f(\phi_5(kh))S(k,h)(\phi_5^{-1}(t))
\right|\\
&=\left|
F(x) -
\sum_{k=-M}^N F(kh)S(k,h)(x)
\right|\\
&\leq \left|
F(x) -
\sum_{k=-\infty}^{\infty} F(kh)S(k,h)(x)
\right|
 +
\left|
\sum_{k=-\infty}^{-M-1} F(kh)S(k,h)(x)
+\sum_{k=N+1}^{\infty} F(kh)S(k,h)(x)
\right|.
\end{align*}
The first and second terms are called the discretization and truncation
errors, respectively.
The following function space is used to bound the discretization error.

\begin{definition}
Let $d$ be a positive constant, and
let $\domD_d(\epsilon)$ be a rectangular domain defined
for $0<\epsilon<1$ by
\[
\domD_d(\epsilon)
= \{\zeta\in\mathbb{C}:|\Re\zeta|<1/\epsilon,\, |\Im\zeta|<d(1-\epsilon)\}.
\]
Then, $\mathbf{H}^1(\domD_d)$ denotes the family of all analytic functions $F$
on $\domD_d$ such that the norm $\mathcal{N}_1(F,d)$ is finite, where
\[
\mathcal{N}_1(F,d)
=\lim_{\epsilon\to 0}\oint_{\partial \domD_d(\epsilon)} |F(\zeta)||\mathrm{d}\zeta|.
\]
\end{definition}

For this function space,
the discretization error is estimated as follows.

\begin{theorem}[Stenger~{\cite[Theorem~3.1.3]{stenger93:_numer}}]
\label{thm:discretization-error}
Let $F\in\mathbf{H}^1(\domD_d)$, and let $h>0$. Then,
\[
\sup_{x\in\mathbb{R}}\left|
F(x) - \sum_{k=-\infty}^{\infty}F(kh)S(k,h)(x)
\right|
\leq \frac{\mathcal{N}_1(F,d)}{\pi d(1 - \E^{-2\pi d/h})}\E^{-\pi d/h}.
\]
\end{theorem}

We use the following lemma,
which completes the estimation of the discretization error.
The proof is provided in Section~\ref{subsec:discretization-error}.

\begin{lemma}
\label{lem:bound-None}
Assume that the assumptions of Theorem~\ref{thm:New1} are satisfied.
Then, the function $F(\zeta)=f(\phi_5(\zeta))$
belongs to $\mathbf{H}^1(\domD_d)$, and $\mathcal{N}_1(F,d)$ is bounded as
\[
 \mathcal{N}_1(F,d)
\leq \frac{2 C_{\mathrm{D}}}{\pi\cos d},
\]
where $C_{\mathrm{D}}$ is the constant defined by~\eqref{eq:C_D}.
\end{lemma}

In addition, we bound the truncation error as follows.
The proof is provided in Section~\ref{subsec:truncation-error}.

\begin{lemma}
\label{lem:truncation-error}
Assume that the assumptions of
Theorem~\ref{thm:New1} are satisfied.
Then, putting $F(x)=f(\phi_5(x))$, we have
\[
\sup_{x\in\mathbb{R}}
\left|
\sum_{k=-\infty}^{-M-1}
F(kh)S(k,h)(x)
+\sum_{k=N+1}^{\infty}
F(kh)S(k,h)(x)
\right|
\leq \frac{C_{\mathrm{T}}}{\pi d} \E^{-\pi d n/\log(2 d n/\mu)},
\]
where $C_{\mathrm{T}}$ is the constant defined by~\eqref{eq:C_T}.
\end{lemma}

Put $q=2dn/\mu$.
Combining Theorem~\ref{thm:discretization-error}
with Lemmas~\ref{lem:bound-None} and~\ref{lem:truncation-error},
and using the choice of $h$ in~\eqref{eq:Def-h-DE},
we obtain the following estimate.
The lower bound on $n$ in Theorem~\ref{thm:New1} implies $q\geq\E$.
Since $q/\log q$ is monotonically increasing for $q\geq\E$,
$q/\log q\geq\E$.  Therefore,
\begin{align*}
  \left|
  F(x)
- \sum_{k=-M}^N F(kh) S(k,h)(x)
\right|
&\leq \frac{2C_{\mathrm{D}} }{\pi d (1 - \E^{-2\pi d/h})\pi\cos d}\E^{-\pi d/h}
+ \frac{C_{\mathrm{T}}}{\pi d}\E^{-\pi d n/\log(2 d n/\mu)}\\
&= \frac{1}{\pi d}\left\{
\frac{2 C_{\mathrm{D}}}{(1-\E^{-\pi \mu q/\log q})\pi\cos d}
+ C_{\mathrm{T}}
\right\}
\E^{-\pi d n/\log(2 d n/\mu)}\\
&\leq \frac{1}{\pi d}\left\{
\frac{2 C_{\mathrm{D}}}{(1-\E^{-\pi \mu \E})\pi\cos d}
+ C_{\mathrm{T}}
\right\}
\E^{-\pi d n/\log(2 d n/\mu)},
\end{align*}
from which we obtain the desired error bound.
This completes the proof of Theorem~\ref{thm:New1}.

\subsection{Useful inequality on $\domD_d^{+}$}
\label{subsec:domDdplus}

Here, the following lemma is prepared.
This lemma was previously stated in an existing study~\cite{okayama22:_doubl},
but the proof was omitted due to space restrictions.
For readers' convenience,
its explicit proof is provided.
Note that $\overline{\domD}$
denotes the closure of $\domD$.
Note that although $1+\E^{\pi\sinh\zeta}$ has no zeros in
$\domD_{\pi/2}$, it may vanish at some points on the boundary.
Nevertheless, the function
\[
 \zeta\mapsto\frac{1}{\log(1+\E^{\pi\sinh\zeta})}
\]
admits a continuous extension to each such boundary point,
with value zero.  In what follows, this function is understood
in this extended sense on $\overline{\domD_{\pi/2}}$.

\begin{lemma}[Okayama~{\cite[Lemma~7]{okayama22:_doubl}}]
It holds for all $\zeta\in\overline{\domD_{\pi/2}^{+}}$ that
\begin{equation}
 \left|\frac{1}{\log(1+\E^{\pi\sinh\zeta})}\right|
\leq \frac{1}{\log 2}.
\label{eq:domDdplus}
\end{equation}
\end{lemma}
\begin{proof}
Let $x = \Re(\pi\sinh\zeta)$ and $y = \Im(\pi\sinh\zeta)$.
Note that $x\geq 0$ because $\zeta\in\overline{\domD_{\pi/2}^{+}}$.
We show
\[
 \left|\log(1 + \E^{x + \I y})\right|\geq \log 2
\]
in the following three cases: (i) $0\leq x\leq \log 3$ and $|y|\leq \pi$,
(ii) $0\leq x\leq \log 3$ and $|y|>\pi$, and (iii) $x> \log 3$.

\begin{enumerate}
 \item[(i)] $0\leq x\leq \log 3$ and $|y|\leq \pi$.
By the maximum modulus principle applied to
$1/\log(1 + \E^{x + \I y})$, it suffices to consider the boundary of
the rectangle region ($0\leq x\leq \log 3$ and $|y|\leq \pi$).
Indeed, $\log(1+\E^{x+\I y})\neq0$ in the interior of the rectangle, because
$\log(1+\E^{x+\I y})=0$ would imply $\E^{x+\I y}=0$.
On the left edge ($x = 0$ and $|y|< \pi$), we have
\begin{equation*}
 \left|\log(1 + \E^{x + \I y})\right|
= \left|\log(1 + \E^{\I y})\right|
= \left|\log(2\cos(y/2) \cdot \E^{\I (y/2)})\right|
= \sqrt{\{\log(2\cos(y/2))\}^2 + (y/2)^2}.
\end{equation*}
As shown in the existing lemma~\cite[Lemma~4.6]{tomoaki21:_new},
the right-hand side attains its minimum at $y=0$, i.e.,
\[
\sqrt{\{\log(2\cos(y/2))\}^2 + (y/2)^2}
\geq \sqrt{\{\log(2\cos( 0/2))\}^2 + (0/2)^2} = \log 2.
\]
At the corner points $y=\pm\pi$, the same estimate follows by taking
the limit $y\to\pm\pi$.
On the top/bottom edges ($0< x < \log 3$ and $y=\pm\pi$),
we have
\[
 \left|\log(1 + \E^{x + \I y})\right|
=\sqrt{\{\log|1 - \E^x|\}^2 + \{\arg(1 - \E^x)\}^2}
=\sqrt{\{\log(\E^x - 1)\}^2 + \pi^2}
\geq \sqrt{0^2 + \pi^2} > \log 2.
\]
On the right edge ($x=\log 3$ and $|y|\leq\pi$),
we have
\[
 \left|\log(1 + \E^{x + \I y})\right|
=\sqrt{\{\log|1+\E^{x+\I y}|\}^2 + \{\arg(1 + \E^{x + \I y})\}^2}
\geq \sqrt{\{\log(|\E^{x+\I y}| - 1)\}^2 + 0^2}
=\log(\E^x - 1)= \log 2.
\]
Thus, $|\log(1 + \E^{x+\I y})|\geq \log 2$ holds
throughout this rectangle region.
 \item[(ii)] $0\leq x\leq \log 3$ and $|y|> \pi$.
First, let $x>0$ and put $\theta(x,y)=\Im[\log(1+\E^{x+\I y})]$.
Then,
\[
  \frac{\uppartial}{\uppartial y} \theta(x,y)
 =\Im\left[\frac{\I \E^{x+\I y}}{1 + \E^{x+\I y}}\right]
 =\Im\left[\frac{-\E^x\sin y+\I (\E^{2x}+\E^x\cos y)}{|1+\E^{x+\I y}|^2}\right]
 =\frac{\E^x(\E^x+\cos y)}
 {|1+\E^{x+\I y}|^2}>0.
\]
Because $\theta(x,\pi)=\pi$ and $\theta(x,-\pi)=-\pi$,
we have $|\theta(x,y)|>\pi$ for $|y|>\pi$. Therefore,
\[
 |\log(1+\E^{x+\I y})|
 \geq|\theta(x,y)|>\pi>\log 2.
\]
For $x=0$ and $1+\E^{\I y}\neq0$, the desired estimate follows
by letting $x\to 0+$.
At points where $1+\E^{\I y}=0$,
the desired inequality follows directly from
the continuous extension defined above.
 \item[(iii)] $x>\log 3$. In this case, we have
\[
 |\log(1 + \E^{x + \I y})|
=\sqrt{\{\log|1+\E^{x+\I y}|\}^2 + \{\arg(1 + \E^{x + \I y})\}^2}
\geq \sqrt{\{\log(|\E^{x+\I y}| - 1)\}^2 + 0^2}
=\log(\E^x - 1)\geq \log 2.
\]
\end{enumerate}
This completes the proof.
\end{proof}

\subsection{Useful inequalities on $\domD_d^{-}$}
\label{subsec:domDdminus}

Here, two lemmas
(Lemmas~\ref{lem:bound-by-1-minus-log2} and~\ref{lem:bound-log-by-exp})
are prepared.
The first lemma (Lemma~\ref{lem:bound-by-1-minus-log2})
was previously stated in an existing study~\cite{okayama22:_doubl},
but the proof was omitted due to space restrictions.
For readers' convenience,
its explicit proof is provided.
The second lemma (Lemma~\ref{lem:bound-log-by-exp})
is a new result.
To assist the proof, Proposition~\ref{prop:g-decrease-increase}
is prepared.

\begin{lemma}[Okayama~{\cite[Lemma~5]{okayama22:_doubl}}]
\label{lem:bound-by-1-minus-log2}
It holds for all $\zeta\in\overline{\domD_{\pi/2}^{-}}$ that
\begin{equation}
 \frac{1}{|-1 + \log(1 + \E^{\pi\sinh\zeta})|}
\leq \frac{1}{1 - \log 2}.
\label{eq:bound-by-1-minus-log2}
\end{equation}
At boundary points where $1+\E^{\pi\sinh\zeta}=0$, the left-hand side
is understood by continuous extension, with value zero.
\end{lemma}
\begin{proof}
By the definition of $\log z$, we have
\[
  \frac{1}{\left|-1 + \log(1 + \E^{\pi\sinh\zeta})\right|}=
\frac{1}{\left|-1 + \log|1 + \E^{\pi\sinh\zeta}| + \I\arg(1 + \E^{\pi\sinh\zeta})\right|}
\leq \frac{1}{\left|-1 + \log|1 + \E^{\pi\sinh\zeta}| + \I \cdot 0\right|}.
\]
Let $x = \Re(\pi\sinh\zeta)$ and $y = \Im(\pi\sinh\zeta)$.
Note that $x\leq 0$ because $\zeta\in\overline{\domD_{\pi/2}^{-}}$.
From the following estimate
\[
 |1 + \E^{x + \I y}|
\leq 1 + |\E^{x + \I y}|
= 1 + \E^x
\leq 1 + \E^0
< \E,
\]
we have $\log|1 + \E^{x + \I y}| < 1$. Therefore,
we have
\[
\frac{1}{\left|-1 + \log|1 + \E^{x + \I y}|\right|}
=
 \frac{1}{1 - \log|1 + \E^{x + \I y}|}
\leq \frac{1}{1 - \log(1 + |\E^{x + \I y}|)}
= \frac{1}{1 - \log(1 + \E^x)}
\leq \frac{1}{1 - \log(1 + \E^0)}
=\frac{1}{1 - \log 2}.
\]
This completes the proof.
\end{proof}

\begin{proposition}
\label{prop:g-decrease-increase}
Let $g(x,y)$ be a function defined by
\begin{equation}
g(x,y)=
\begin{cases}
\dfrac{\cosh(x)-\cos(y)}{x^2+y^2} &((x,y)\neq(0,0)),\\
\dfrac12 &((x,y)=(0,0)).
\end{cases}
\label{eq:g-x-y}
\end{equation}
Then, for arbitrary fixed $y\in\mathbb{R}$,
$g(x,y)$ monotonically decreases for $x<0$ and
monotonically increases for $x>0$.
\end{proposition}
\begin{proof}
Differentiating $g$ with respect to $x$,
we have
\[
 g_x(x,y)
= \frac{\tilde{g}(x,y)}{(x^2 + y^2)^2},
\]
where
\[
 \tilde{g}(x,y)= 2x\cos y - 2x\cosh x + (x^2 + y^2)\sinh x.
\]
Differentiating $\tilde{g}$ with respect to $x$, we have
\begin{align*}
 \tilde{g}_x(x,y)
&= 2 \cos y + (x^2 + y^2 - 2)\cosh x,\\
  \tilde{g}_{xx}(x,y)
&= 2 x\cosh x + (x^2 + y^2 - 2)\sinh x,\\
 \tilde{g}_{xxx}(x,y)
&= 4x\sinh x + (x^2 + y^2)\cosh x.
\end{align*}
Because $\tilde{g}_{xxx}(x,y)\geq 0$,
$\tilde{g}_{xx}(x,y)$
is a monotonically increasing function with respect to $x$.
From this and using $\tilde{g}_{xx}(0,y) = 0$,
we have $\tilde{g}_{xx}(x,y)<0$
for $x<0$ and $\tilde{g}_{xx}(x,y)>0$
for $x>0$.
Therefore, $\tilde{g}_{x}(x,y)$
monotonically decreases for $x<0$ and
monotonically increases for $x>0$.
From this and using
$\tilde{g}_x(0,y) = 4\{(y/2)^2 - \sin^2(y/2)\}\geq 0$,
we have $\tilde{g}_x(x,y)\geq 0$
for all $x\in\mathbb{R}$, which implies that
$\tilde{g}(x,y)$ is a monotonically increasing function with respect to $x$.
From this and using $\tilde{g}(0,y)=0$,
we have $g_x(x,y)<0$
for $x<0$ and $g_x(x,y)>0$
for $x>0$.
This completes the proof.
\end{proof}

\begin{lemma}
\label{lem:bound-log-by-exp}
It holds for all $\zeta\in\domD_{\pi/2}^{-}$ that
\begin{equation}
 \left|
\frac{\log(1 + \E^{\pi\sinh\zeta})}{1 + \log(1 + \E^{\pi\sinh\zeta})}
\cdot\frac{\E^{-L} + \E^{\pi\sinh\zeta}}{\E^{\pi\sinh\zeta}}
\right|\leq 1,
\label{eq:bound-log-by-exp}
\end{equation}
where $L = \log(\E/(\E - 1))$.
\end{lemma}
\begin{proof}
Putting $\eta = \log(1 + \E^{\pi\sinh\zeta})$, we have
\begin{align*}
  \left|
\frac{\log(1 + \E^{\pi\sinh\zeta})}{1 + \log(1 + \E^{\pi\sinh\zeta})}
\cdot\frac{\E^{-L} + \E^{\pi\sinh\zeta}}{\E^{\pi\sinh\zeta}}
\right|
&=\left|
\frac{\eta}{1 + \eta}\cdot\frac{\frac{\E-1}{\E} + \E^{\eta} - 1}{\E^{\eta} - 1}
\right|
=\frac{1}{\E}\left|
\frac{\eta}{\E^{\eta} - 1}\cdot\frac{\E^{\eta+1} - 1}{\eta+1}
\right|.
\end{align*}
Set $x = \Re\eta$ and $y = \Im\eta$.
Substituting $\eta = x + \I y$, we have
\begin{align*}
\frac{1}{\E}\left|
\frac{\eta}{\E^{\eta} - 1}\cdot\frac{\E^{\eta+1} - 1}{\eta+1}
\right|
&=\sqrt{\frac{1}{\E}\cdot\frac{x^2 + y^2}{\cosh(x) - \cos y}
\cdot\frac{\cosh(x+1) - \cos y}{(x+1)^2 + y^2}}
=\sqrt{\frac{1}{\E}\cdot\frac{g(x+1,y)}{g(x,y)}},
\end{align*}
where $g(x,y)$ is defined by~\eqref{eq:g-x-y}.
Here, we used
$|\E^{x+\I y}-1|^2=2\E^x(\cosh x-\cos y)$.
Because $\Re\zeta < 0$ and $|\Im\zeta|<\pi/2$,
we have $\Re(\sinh\zeta)=\sinh(\Re\zeta)\cos(\Im\zeta)<0$, and hence
$|\E^{\pi\sinh\zeta}|< 1$,
which implies $\Re(1 + \E^{\pi\sinh\zeta}) > 0$.
Consequently,
we have $|y| = |\arg(1 + \E^{\pi\sinh\zeta})|\leq \pi/2$.
Moreover, by the triangle inequality,
\[
 x=\log|1+\E^{\pi\sinh\zeta}|
 \leq\log(1+|\E^{\pi\sinh\zeta}|)
 <\log 2.
\]
We consider two cases for $x$: (i) $x\leq -1$ and
(ii) $-1<x<\log 2$.

\begin{enumerate}
 \item[(i)]
 $x\leq -1$. In this case, $g(x,y)\geq g(x+1,y)$ holds
from Proposition~\ref{prop:g-decrease-increase}.
Therefore, we have
\[
  \sqrt{\frac{1}{\E}\cdot\frac{g(x+1,y)}{g(x,y)}}
\leq \frac{1}{\sqrt{\E}}<1.
\]
\item[(ii)] $-1<x<\log 2$. Put $A=1+\log 2$.
In this case,
$g(x+1,y)\leq g(A,y)$
and $g(x,y)\geq g(0,y)$ hold
from Proposition~\ref{prop:g-decrease-increase}.
Furthermore, using
\begin{align*}
 g(A,y)&=\frac{\cosh A-\cos y}{A^2+y^2}
\leq \frac{\cosh A}{A^2},\\
\frac{1}{g(0,y)}&=\frac{y^2}{1 - \cos y}
\leq \frac{(\pi/2)^2}{1 - \cos(\pi/2)} = \frac{\pi^2}{4},
\end{align*}
where the second inequality follows because
$y^2/(1-\cos y)=2\{(y/2)/\sin(y/2)\}^2$, and
$t/\sin t$ is increasing for $t\geq0$.
Therefore,
\[
  \sqrt{\frac{1}{\E}\cdot\frac{g(x+1,y)}{g(x,y)}}
\leq \sqrt{\frac{\pi^2\cosh A}{4\E A^2}}
=\frac{\pi}{2(1+\log 2)}
\sqrt{1+\frac{1}{4\E^2}}.
\]
Since $1/t$ is strictly convex on $[1,2]$, we have
$\log 2=\int_1^2(1/t)\D t>2/3$.
Using this inequality together with $\pi^2<10$ and $\E>2$, we obtain
\begin{align*}
\left\{
\frac{\pi}{2(1+\log 2)}
\sqrt{1+\frac{1}{4\E^2}}
\right\}^2
&<\frac{10}{4(1+2/3)^2}
\left(1+\frac{1}{16}\right)
=\frac{153}{160}<1.
\end{align*}
\end{enumerate}
This completes the proof.
\end{proof}

\subsection{Useful inequalities on $\domD_d$}
\label{subsec:domDd}

Here, we prepare Lemmas~\ref{lem:DE-func-bound}
and~\ref{lem:third-estimate-continuous}.
The former is used on $\domD_d^+$, whereas the latter provides a
continuous positive lower-bound factor on $\domD_d^-$.
To prove Lemma~\ref{lem:third-estimate-continuous}, we prepare
Lemmas~\ref{lem:gudermannian-lower-bound},
\ref{lem:continuous-kappa-algebraic-bound},
and~\ref{lem:continuous-kappa-branch}.

\begin{lemma}[Okayama et al.~{\cite[Lemma~4.22]{Okayama-et-al}}]
\label{lem:DE-func-bound}
Let $x$ and $y$ be real numbers with $|y|<\pi/2$.
Then, we have
\begin{align}
\frac{1}{|1 + \E^{\pi\sinh(x+\I y)}|}
&\leq \frac{1}{(1 + \E^{\pi\sinh(x)\cos y})\cos((\pi/2)\sin y)},
\label{eq:DE-func-bound-plus}\\
\frac{1}{|1 + \E^{-\pi\sinh(x+\I y)}|}
&\leq \frac{1}{(1 + \E^{-\pi\sinh(x)\cos y})\cos((\pi/2)\sin y)}.
\nonumber
\end{align}
\end{lemma}

\begin{lemma}[Sun and Chen~{\cite[Theorem~4]{SunChen2025}}]
\label{lem:gudermannian-lower-bound}
For $t>0$,
\begin{equation*}
 \arcsin(\tanh t)
 >\frac{\pi t}{\sqrt{\pi^2+4t^2}}.
\end{equation*}
\end{lemma}

\begin{lemma}
\label{lem:continuous-kappa-algebraic-bound}
For $r>0$,
\begin{equation}
 4r^2+\frac{\pi^2}{r^2}
 \geq\pi^2+\frac{4r^4}{(1+r)^2}.
 \label{eq:continuous-kappa-algebraic-bound}
\end{equation}
\end{lemma}
\begin{proof}
For $0<r\leq1$, subtracting the right-hand side of
\eqref{eq:continuous-kappa-algebraic-bound} from the left-hand side,
we have
\begin{align*}
 4r^2+\frac{\pi^2}{r^2}
 -\left(\pi^2+\frac{4r^4}{(1+r)^2}\right)
 =4r^2\left(1-\frac{r^2}{(1+r)^2}\right)
 +\pi^2\left(\frac{1}{r^2}-1\right)\geq0.
\end{align*}
For $r\geq1$, the desired inequality is equivalent to
\[
 4r^4(1+2r)\geq\pi^2(r-1)(r+1)^3.
\]
Using $\pi^2<10$, it is enough to show the same inequality with
$\pi^2$ replaced by $10$.  Put
\[
 t=\frac{r-1}{r+1}\in[0,1).
\]
Then, a direct calculation gives
\[
 4r^4(1+2r)-10(r-1)(r+1)^3
 =\frac{4}{(1-t)^5}
 \left\{t^5+7t^4+18t^3+62t^2-27t+3\right\}.
\]
Since
\[
 62t^2-27t+3
 =62\left(t-\frac{27}{124}\right)^2+\frac{15}{248}>0,
\]
the expression on the right-hand side is positive for $0\leq t<1$.
This proves~\eqref{eq:continuous-kappa-algebraic-bound}.
\end{proof}

\begin{lemma}
\label{lem:continuous-kappa-branch}
Let $L$ be a positive constant, and let $d_L$ be the positive constant
defined by~\eqref{eq:d_L}.  For $y\in(0,d_L)$, let
$T_{L,y}$, $\tau_{L,y}$, and $a_{L,y}$ be defined by
\eqref{eq:T-L-d}--\eqref{eq:a-L-d}, respectively, with $d$ replaced
by $y$.
Then,
\[
 T_{L,y}>0,\quad \tau_{L,y}>0,\quad
 0<a_{L,y}<\frac{\pi}{2}
 \quad(0<y<d_L),
\]
and
\begin{equation}
 \cos a_{L,y}\leq\tanh\tau_{L,y}
 \qquad(0<y<d_L).
 \label{eq:continuous-kappa-branch}
\end{equation}
\end{lemma}
\begin{proof}
Direct substitution of~\eqref{eq:d_L} gives
$\pi\cos^2d_L-L\sin d_L=0$.  Moreover, the function
$\pi\cos^2y-L\sin y$ is strictly decreasing on $y\in(0,\pi/2)$.
Therefore, $\pi\cos^2y-L\sin y>0$ for $0<y<d_L$, and hence
$T_{L,y}>0$.  Furthermore,
$0<\tau_{L,y}<T_{L,y}$.
Since
\[
 \frac{2T_{L,y}}{\pi\cos y}
 =\frac{\cos y}{\sin y}-\frac{L}{\pi\cos y},
\]
we have
\[
  \frac{L + 2\tau_{L,y}}{\pi\cos y}
< \frac{L + 2T_{L,y}}{\pi\cos y}
 =\frac{\cos y}{\sin y}.
\]
Consequently,
\[
 0<a_{L,y}
 <\frac{\pi}{2}\sin y
 \sqrt{1+\frac{\cos^2y}{\sin^2y}}
 =\frac{\pi}{2}.
\]

Next, we prove~\eqref{eq:continuous-kappa-branch}.
Set
\[
 \xi_0=\frac{L + 2T_{L,y}}{\pi\cos y}
 =\frac{\cos y}{\sin y},
 \qquad
 \xi=\frac{L + 2\tau_{L,y}}{\pi\cos y}.
\]
Then,
\[
 \xi_0-\xi
 =\frac{2(T_{L,y}-\tau_{L,y})}{\pi\cos y}.
\]
Moreover,~\eqref{eq:tau-L-d} with $d$ replaced by $y$ gives
\[
 T_{L,y}-\tau_{L,y}
 =\tau_{L,y}\sqrt{T_{L,y}}.
\]
Let
\[
 R(t)=\sqrt{1+t^2},
\]
which is convex on $\mathbb{R}$.  Since
\[
 a_{L,y}=\frac{\pi}{2}(\sin y)R(\xi),
 \qquad
 R(\xi_0)=\frac{1}{\sin y},
 \qquad
 R'(\xi_0)=\cos y,
\]
the convexity of $R$ gives
\begin{align*}
 \frac{\pi}{2}-a_{L,y}
 &\leq\frac{\pi}{2}(\sin y)R'(\xi_0)(\xi_0-\xi)
  =(\sin y)(T_{L,y}-\tau_{L,y})
  =(\sin y)\tau_{L,y}\sqrt{T_{L,y}}.
\end{align*}

We next derive an upper bound for the last expression.  From
\eqref{eq:T-L-d} with $d$ replaced by $y$,
\begin{equation}
 L+2T_{L,y}
 =\pi\left(\frac{1}{\sin y}-\sin y\right).
 \label{eq:L-plus-2T}
\end{equation}
Solving this equation for $\sin y$ and using $L>0$, we obtain
\begin{equation}
 \sin y
 =\frac{2\pi}
 {\sqrt{(L+2T_{L,y})^2+4\pi^2}+L+2T_{L,y}}
 <\frac{\pi}
 {\sqrt{T_{L,y}^2+\pi^2}+T_{L,y}}.
 \label{eq:continuous-kappa-sine-bound}
\end{equation}

It remains to relate the right-hand side of
\eqref{eq:continuous-kappa-sine-bound} to
Lemmas~\ref{lem:gudermannian-lower-bound} and
\ref{lem:continuous-kappa-algebraic-bound}.
Now set $T=T_{L,y}$, $\tau=\tau_{L,y}$, and $s=\sqrt T$.
Then,~\eqref{eq:tau-L-d} with $d$ replaced by $y$ gives
$\tau=s^2/(1+s)$.
By Lemma~\ref{lem:continuous-kappa-algebraic-bound} and the triangle
inequality in $\mathbb{R}^2$,
\begin{align*}
 \sqrt{\pi^2+4\tau^2}
  =\sqrt{\pi^2+\frac{4s^4}{(1+s)^2}}
  \leq\sqrt{4s^2+\frac{\pi^2}{s^2}}
  \leq s+\sqrt{s^2+\frac{\pi^2}{s^2}}.
\end{align*}
Indeed, the last inequality follows from
\[
 \left\|\left(2s,\frac{\pi}{s}\right)\right\|_2
 \leq\|(s,0)\|_2
 +\left\|\left(s,\frac{\pi}{s}\right)\right\|_2.
\]
Equivalently,
\[
 \frac{\sqrt T}{\sqrt{T^2+\pi^2}+T}
 \leq\frac{1}{\sqrt{\pi^2+4\tau^2}}.
\]
Consequently,~\eqref{eq:continuous-kappa-sine-bound} and
Lemma~\ref{lem:gudermannian-lower-bound} give
\begin{align*}
(\sin y)\tau\sqrt T
  <\frac{\pi\tau\sqrt T}{\sqrt{T^2+\pi^2}+T}
  \leq\frac{\pi\tau}{\sqrt{\pi^2+4\tau^2}}
  <\arcsin(\tanh\tau).
\end{align*}
Hence,
\[
 a_{L,y}
 \geq\frac{\pi}{2}-\arcsin(\tanh\tau_{L,y})
 =\arccos(\tanh\tau_{L,y}),
\]
which proves~\eqref{eq:continuous-kappa-branch}.
\end{proof}

\begin{remark}
The particular choice
$\tau_{L,y}=T_{L,y}/(1+\sqrt{T_{L,y}})$ was made to establish
the inequality~\eqref{eq:continuous-kappa-branch} by a simple explicit
argument.  If $a_{L,y}$ is regarded as a function of $\tau_{L,y}$,
the two lower-bound factors used in the subsequent proof of
Lemma~\ref{lem:third-estimate-continuous} are
$\tanh\tau_{L,y}$ and $\cos a_{L,y}$.
The former is increasing with respect to $\tau_{L,y}$, whereas the
latter is decreasing.  Therefore, choosing $\tau_{L,y}$ so that
$\cos a_{L,y}=\tanh\tau_{L,y}$ would balance these two factors and
could yield a sharper computable error bound.  Since such a choice is
defined only implicitly and does not appear to admit a closed-form
expression, its systematic use is left for future work.
\end{remark}

\begin{lemma}
\label{lem:third-estimate-continuous}
Under the assumptions and notation of
Lemma~\ref{lem:continuous-kappa-branch}, for
$y\in(-d_L,d_L)$, define
\begin{equation}
 \kappa_{L,y}
 =
 \begin{cases}
 1 & (y=0),\\[1ex]
 \cos a_{L,|y|}
 & (0<|y|<d_L).
 \end{cases}
 \label{eq:continuous-kappa}
\end{equation}
Then, $\kappa_{L,y}$ is continuous and positive for
$y\in(-d_L,d_L)$.
Moreover, for all $x\in\mathbb{R}$ and $y\in(-d_L,d_L)$,
\begin{equation}
 \frac{1}{|1+\E^{-L-\pi\sinh(x+\I y)}|}
 \leq
 \frac{1}
 {(1+\E^{-L-\pi\sinh(x)\cos y})
 \kappa_{L,y}}.
 \label{eq:third-estimate-continuous}
\end{equation}
\end{lemma}
\begin{proof}
Lemma~\ref{lem:continuous-kappa-branch} shows that
$\tanh\tau_{L,y}>0$ and $\cos a_{L,y}>0$ for $0<y<d_L$.
Therefore, $\kappa_{L,y}>0$.
Moreover, as $y\to0^+$, we have
$T_{L,y}=\OO(1/y)$, $\tau_{L,y}=\OO(1/\sqrt{y})$, and hence
$a_{L,y}=\OO(\sqrt{y})$.  Thus,
$\kappa_{L,y}=\cos a_{L,|y|}\to1=\kappa_{L,0}$ as $y\to0$.
Continuity away from the origin follows directly from the definitions,
so $\kappa_{L,y}$ is continuous on $(-d_L,d_L)$.

It remains to prove~\eqref{eq:third-estimate-continuous}.
For $y=0$, its two sides are equal because
\[
 |1+\E^{-L-\pi\sinh x}|
 =1+\E^{-L-\pi\sinh x}
\]
and $\kappa_{L,0}=1$.
For $y\ne0$, all the quantities involved are even with respect to
$y$, so it suffices to consider $0<y<d_L$.
Let $x_{L,y}=\arsinh[L/(\pi\cos y)]$.
Since $L>0$ and $0<y<d_L<\pi/2$, we have $x_{L,y}>0$.
Define
\[
 A(x,y)=\frac{\pi}{2}\cosh x\sin y,
 \qquad
 B(x,y)=\frac{\pi}{2}(\sinh x+\sinh x_{L,y})\cos y,
 \qquad
 G(x,y)=1-\frac{\sin^2A(x,y)}{\cosh^2B(x,y)}.
\]
Set
\[
 \rho=\E^{-L-\pi\sinh(x)\cos y},
 \qquad
 \vartheta=\pi\cosh(x)\sin y.
\]
Then,
\begin{align*}
 |1+\rho\E^{-\I\vartheta}|^2
 &=(1+\rho)^2-4\rho\sin^2(\vartheta/2)\\
 &=(1+\rho)^2\left\{1-
 \frac{\sin^2(\vartheta/2)}
 {\cosh^2[(L/2)+(\pi/2)\sinh(x)\cos y]}
 \right\}\\
 &=(1+\rho)^2G(x,y),
\end{align*}
where we used
\[
 \frac{4\rho}{(1+\rho)^2}
 =\frac{1}{\cosh^2[(L/2)+(\pi/2)\sinh(x)\cos y]}
\]
and
\[
 B(x,y)=\frac{L}{2}+\frac{\pi}{2}\sinh(x)\cos y.
\]
Consequently,
\begin{equation}
\frac{1}{|1+\E^{-L-\pi\sinh(x+\I y)}|}
=\frac{1}{(1 + \E^{-L-\pi\sinh(x)\cos y})\sqrt{G(x,y)}}.
\label{eq:continuous-kappa-identity}
\end{equation}
Thus, it remains to prove
$G(x,y)\geq \kappa_{L,y}^2$.

If $|B(x,y)|\geq\tau_{L,y}$, then
\[
 G(x,y)
 \geq1-\frac{1}{\cosh^2\tau_{L,y}}
 =\tanh^2\tau_{L,y}
 \geq\cos^2a_{L,y}
 =\kappa_{L,y}^2,
\]
where the second inequality follows from
Lemma~\ref{lem:continuous-kappa-branch}.
On the other hand, if $|B(x,y)|<\tau_{L,y}$, then
\[
 |\sinh x+
 \sinh x_{L,y}|<\frac{2\tau_{L,y}}{\pi\cos y},
\]
and combining this with
$|\sinh x|-\sinh(x_{L,y})\leq |\sinh x + \sinh x_{L,y}|$,
we have
\[
 |\sinh x|<\sinh x_{L,y}
 +\frac{2\tau_{L,y}}{\pi\cos y}
 = \frac{L + 2\tau_{L,y}}{\pi\cos y}.
\]
Therefore, $0<A(x,y)<a_{L,y}<\pi/2$, and
\[
 G(x,y)
 \geq1-\sin^2a_{L,y}
 =\cos^2a_{L,y}
 =\kappa_{L,y}^2.
\]
Combining this estimate with the identity
\eqref{eq:continuous-kappa-identity}
proves~\eqref{eq:third-estimate-continuous}.
\end{proof}

\subsection{Bound of the discretization error (proof of Lemma~\ref{lem:bound-None})}
\label{subsec:discretization-error}

To prove Lemma~\ref{lem:bound-None}, we prepare the following lemma.

\begin{lemma}
\label{lem:a_L_y-increase}
Under the assumptions and notation of
Lemma~\ref{lem:continuous-kappa-branch},
$a_{L,y}$ is strictly increasing
with respect to $y$ on $(0,d_L)$.
Moreover, $\cos a_{L,y}$ is strictly decreasing
with respect to $y$ on $(0,d_L)$.
\end{lemma}
\begin{proof}
By Lemma~\ref{lem:continuous-kappa-branch},
$T_{L,y}>0$ and $0<a_{L,y}<\pi/2$ for $0<y<d_L$.
Put $s=\sqrt{T_{L,y}}>0$.
From~\eqref{eq:L-plus-2T}, we have
\[
  L + 2s^2 = L + 2 T_{L,y} =
   \frac{\pi\cos^2 y}{\sin y}.
\]
Moreover,
\[
L + 2\tau_{L,y} = L + 2 \frac{s^2}{1 + s}
= (L + 2s^2) - \frac{2s^3}{1+s}
= (L + 2s^2)\left(1 - r\right)
= \frac{\pi\cos^2 y}{\sin y}\left(1 - r\right),
\]
where $r = 2s^3 / \{(L+2s^2)(1+s)\}$.
Since $(L+2s^2)(1+s) - 2s^3 = L(1+s) + 2s^2 > 0$,
$0<r<1$ holds. Squaring $a_{L,y}$, we have
\begin{equation*}
  a_{L,y}^2
  =\frac{\pi^2}{4}\sin^2 y
  \left\{1 + \left(\frac{L + 2\tau_{L,y}}{\pi\cos y}\right)^2\right\}
  = \frac{\pi^2}{4}\sin^2 y
  \left\{1 + \frac{\cos^2 y}{\sin^2 y}(1 - r)^2\right\}
  = \frac{\pi^2}{4}
  \left\{1 -  (2r - r^2)\cos^2 y\right\}.
% \label{eq:a_L_y-square}
\end{equation*}
Since $a_{L,y}> 0$ and $\diff(a_{L,y}^2)/\diff y = (2a_{L,y})\diff a_{L,y}/\diff y$,
the signs of $\diff(a_{L,y}^2)/\diff y$ and $\diff(a_{L,y})/\diff y$ match.
Therefore, we show $\diff(a_{L,y}^2)/\diff y>0$ below.
Differentiating $a_{L,y}^2$ with respect to $y$ yields
\begin{align*}
  \frac{\diff a_{L,y}^2}{\diff y}
 &= \frac{\pi^2}{4}
  \left\{
     (2r - r^2)\cdot 2\cos y\sin y
    -\frac{\diff}{\diff y}(2r - r^2)\cdot \cos^2 y 
  \right\}
  =\frac{\pi^2}{4}\left\{
     r(2 - r)\sin 2y - 2(1 - r)\frac{\diff r}{\diff y}\cos^2 y
    \right\}.
\end{align*}
Noting $0<r<1$ and $0<y<d_L<\pi/2$,
we obtain $\diff(a_{L,y}^2)/\diff y>0$ if we show $\diff r/\diff y < 0$.
Differentiating $r$ with respect to $s$ yields
\[
  \frac{\diff r}{\diff s}=\frac{2s^2(2s^2+2Ls+3L)}{(L+2s^2)^2(1+s)^2}>0.
\]
Differentiating $s$ with respect to $y$ yields
\[
  \frac{\diff s}{\diff y}
  =\frac{\diff}{\diff y}\sqrt{T_{L,y}}
  =\frac{1}{2\sqrt{T_{L,y}}}\cdot\frac{\diff T_{L,y}}{\diff y}
  =\frac{1}{2s}\left(-\frac{\pi(1+\sin^2 y)\cos y}{2\sin^2 y}\right) < 0.
\]
Thus, we have established the remaining condition
$\diff r/\diff y = (\diff r/\diff s)(\diff s/\diff y)<0$,
thereby completing the proof that $a_{L,y}$ is strictly increasing.
Moreover, because the cosine function is strictly decreasing on
$(0,\pi/2)$ and $a_{L,y}\in(0,\pi/2)$, $\cos a_{L,y}$ is strictly
decreasing.
\end{proof}

Using this result,
Lemma~\ref{lem:bound-None} is shown as follows.

\begin{proof}
Because $f$ is analytic in $\phi_5(\domD_d)$,
$F(\cdot)=f(\phi_5(\cdot))$ is analytic in $\domD_d$.
Therefore, it remains to show
$\mathcal{N}_1(F,d)\leq 2 C_{\mathrm{D}}/(\pi\cos d)$.
Note that
$\phi_5(\zeta) = \log(1 + \E^{\pi\sinh\zeta}) - \{1/\log(1 + \E^{\pi\sinh\zeta})\}$.
For $\zeta\in\domD_d^{+}$, we have
$\phi_5(\zeta)\in\phi_5(\domD_d^{+})$ by definition. Therefore,
from~\eqref{eq:f-bound-plus} and~\eqref{eq:domDdplus}, we have
\begin{equation}
 |F(\zeta)|
\leq K_{+}\left|\E^{-\phi_5(\zeta)}\right|^{\beta}
= K_{+}\left|\E^{1/\log(1+\E^{\pi\sinh\zeta})}\right|^{\beta}
\left|\frac{1}{1+\E^{\pi\sinh\zeta}}\right|^{\beta}
\leq K_{+}\left(\E^{1/\log 2}\right)^{\beta}
\frac{1}{|1+\E^{\pi\sinh\zeta}|^{\beta}}.
\label{eq:bound-F-zeta-plus}
\end{equation}
Furthermore, for $\zeta\in\domD_d^{-}$, we have
$\phi_5(\zeta)\in\phi_5(\domD_d^{-})$. Hence,
from~\eqref{eq:f-bound-minus},~\eqref{eq:bound-by-1-minus-log2}
and~\eqref{eq:bound-log-by-exp}, we have
\begin{align}
 |F(\zeta)|
&\leq K_{-}\frac{1}{|\phi_{5}(\zeta)|^{\alpha}}\nonumber\\
& = K_{-}\frac{1}{|- 1 + \log(1 + \E^{\pi\sinh\zeta})|^{\alpha}}
 \left|
\frac{\log(1 + \E^{\pi\sinh\zeta})}{1+\log(1 + \E^{\pi\sinh\zeta})}
\right|^{\alpha}\nonumber\\
&\leq  \frac{K_{-}}{(1 - \log 2)^{\alpha}}
\frac{1}{|1 + \E^{-L-\pi\sinh\zeta}|^{\alpha}}.
\label{eq:bound-F-zeta-minus}
\end{align}
By definition, $\mathcal{N}_1(F,d)$ is expressed as
\begin{align}
  \mathcal{N}_1(F,d)
&=\lim_{\epsilon\to 0}
\left\{
\int_{-1/\epsilon}^{1/\epsilon}|F(x - \I d(1-\epsilon))|\D x
+\int_{-d(1-\epsilon)}^{d(1-\epsilon)}|F(1/\epsilon + \I y)|\D y
\right.\nonumber\\
&\quad\quad\left.
+\int_{-1/\epsilon}^{1/\epsilon}|F(x + \I d(1-\epsilon))|\D x
+\int_{-d(1-\epsilon)}^{d(1-\epsilon)}|F(-1/\epsilon + \I y)|\D y
\right\}.
\label{eq:N-one-by-definition}
\end{align}
Using~\eqref{eq:DE-func-bound-plus}
and~\eqref{eq:bound-F-zeta-plus},
we can bound the second term as
\begin{align*}
 \int_{-d(1-\epsilon)}^{d(1-\epsilon)}|F(1/\epsilon + \I y)|\D y
&\leq K_{+}\left(\E^{1/\log 2}\right)^{\beta}
\int_{-d(1-\epsilon)}^{d(1-\epsilon)}
\frac{1}{|1+\E^{\pi\sinh(1/\epsilon + \I y)}|^{\beta}}\D y\\
&\leq K_{+}\left(\E^{1/\log 2}\right)^{\beta}
\int_{-d(1-\epsilon)}^{d(1-\epsilon)}
\frac{1}{(1+\E^{\pi\sinh(1/\epsilon)\cos y})^{\beta}\cos^{\beta}((\pi/2)\sin y)}\D y\\
&\leq \frac{K_{+}\left(\E^{1/\log 2}\right)^{\beta}}{(1+\E^{\pi\sinh(1/\epsilon)\cos d})^{\beta}\cos^{\beta}((\pi/2)\sin d)}
\int_{-d(1-\epsilon)}^{d(1-\epsilon)}\D y,
\end{align*}
from which we have
\[
 \lim_{\epsilon\to 0}
 \int_{-d(1-\epsilon)}^{d(1-\epsilon)}|F(1/\epsilon + \I y)|\D y
 = 0.
\]
By Lemma~\ref{lem:a_L_y-increase} and the evenness of $\kappa_{L,y}$,
we have $\kappa_{L,y}\geq\kappa_{L,d}>0$ for $|y|\leq d$.
Using this estimate together with $\cos y\geq\cos d$ for $|y|\leq d$,
and applying~\eqref{eq:third-estimate-continuous} and
\eqref{eq:bound-F-zeta-minus}, we obtain
\begin{align*}
\int_{-d(1-\epsilon)}^{d(1-\epsilon)}
|F(-1/\epsilon+\I y)|\D y
\leq
2dK_{-}
\left\{
\frac{1}
{(1-\log 2)\kappa_{L,d}}
\right\}^{\alpha}
\frac{1}
{\{1+\E^{-L-\pi\sinh(-1/\epsilon)\cos d}\}^{\alpha}}.
\end{align*}
Therefore,
\[
 \lim_{\epsilon\to 0}
 \int_{-d(1-\epsilon)}^{d(1-\epsilon)}|F(-1/\epsilon + \I y)|\D y
 = 0.
\]
Set $d_{\epsilon}=d(1-\epsilon)$, and denote the first and third
terms in~\eqref{eq:N-one-by-definition} by
$H_{\epsilon}^{-}$ and $H_{\epsilon}^{+}$, respectively.
Splitting $H_{\epsilon}^{-}$ at $x=0$ and extending the intervals
to the corresponding half-lines, we have
\[
 H_{\epsilon}^{-}
 \leq \int_{-\infty}^{0}|F(x-\I d_{\epsilon})|\D x
 +\int_{0}^{\infty}|F(x-\I d_{\epsilon})|\D x.
\]
For the first integral, using~\eqref{eq:third-estimate-continuous}
and~\eqref{eq:bound-F-zeta-minus}, we have
\begin{align*}
\int_{-\infty}^{0}|F(x-\I d_{\epsilon})|\D x
&\leq K_{-}
\left\{\frac{1}{1-\log 2}\right\}^{\alpha}
\int_{-\infty}^{0}
\frac{1}{|1+\E^{-L-\pi\sinh(x-\I d_{\epsilon})}|^{\alpha}}\D x\\
&\leq K_{-}
\left\{\frac{1}{1-\log 2}\right\}^{\alpha}
\int_{-\infty}^{0}
\frac{1}{(1+\E^{-L-\pi\sinh(x)\cos d_{\epsilon}})^{\alpha}
\kappa_{L,d_{\epsilon}}^{\alpha}}\D x\\
&=K_{-}
\left\{\frac{1}
{(1-\log 2)\kappa_{L,d_{\epsilon}}}\right\}^{\alpha}
\int_{-\infty}^{0}
\frac{\E^{\pi\alpha\sinh(x)\cos d_{\epsilon}}}
{(\E^{\pi\sinh(x)\cos d_{\epsilon}}+\E^{-L})^{\alpha}}\D x\\
&\leq K_{-}
\left\{\frac{1}
{(1-\log 2)\kappa_{L,d_{\epsilon}}}\right\}^{\alpha}
\int_{-\infty}^{0}
\frac{\cosh(x)\E^{\pi\alpha\sinh(x)\cos d_{\epsilon}}}
{(\E^{-L})^{\alpha}}\D x\\
&=K_{-}
\left\{\frac{1}
{(1-\log 2)\kappa_{L,d_{\epsilon}}}\right\}^{\alpha}
\cdot\frac{(\E^L)^{\alpha}}{\pi\alpha\cos d_{\epsilon}}\\
&=\frac{K_{-}}{\pi\alpha\cos d_{\epsilon}}
\left\{
\frac{\E}
{(1-\log 2)(\E-1)\kappa_{L,d_{\epsilon}}}
\right\}^{\alpha}.
\end{align*}
Here, we used $\E^L=\E/(\E-1)$.
For the second integral, using~\eqref{eq:DE-func-bound-plus}
and~\eqref{eq:bound-F-zeta-plus}, we have
\begin{align*}
\int_{0}^{\infty}|F(x-\I d_{\epsilon})|\D x
&\leq K_{+}\left(\E^{1/\log 2}\right)^{\beta}
\int_{0}^{\infty}
\frac{1}{|1+\E^{\pi\sinh(x-\I d_{\epsilon})}|^{\beta}}\D x\\
&\leq K_{+}\left(\E^{1/\log 2}\right)^{\beta}
\int_{0}^{\infty}
\frac{1}{(1+\E^{\pi\sinh(x)\cos d_{\epsilon}})^{\beta}
\cos^{\beta}((\pi/2)\sin d_{\epsilon})}\D x\\
&=K_{+}
\left\{\frac{\E^{1/\log 2}}
{\cos((\pi/2)\sin d_{\epsilon})}\right\}^{\beta}
\int_{0}^{\infty}
\frac{\E^{-\pi\beta\sinh(x)\cos d_{\epsilon}}}
{(\E^{-\pi\sinh(x)\cos d_{\epsilon}}+1)^{\beta}}\D x\\
&\leq K_{+}
\left\{\frac{\E^{1/\log 2}}
{\cos((\pi/2)\sin d_{\epsilon})}\right\}^{\beta}
\int_{0}^{\infty}
\cosh(x)\E^{-\pi\beta\sinh(x)\cos d_{\epsilon}}\D x\\
&=\frac{K_{+}}{\pi\beta\cos d_{\epsilon}}
\left\{\frac{\E^{1/\log 2}}
{\cos((\pi/2)\sin d_{\epsilon})}\right\}^{\beta}.
\end{align*}
In the same manner, $H_{\epsilon}^{+}$ is bounded by the same
right-hand side.  Therefore,
\begin{align*}
H_{\epsilon}^{-}+H_{\epsilon}^{+}
&\leq \frac{2K_{-}}{\pi\alpha\cos d_{\epsilon}}
\left\{
\frac{\E}
{(1-\log 2)(\E-1)\kappa_{L,d_{\epsilon}}}
\right\}^{\alpha}
+
\frac{2K_{+}}{\pi\beta\cos d_{\epsilon}}
\left\{
\frac{\E^{1/\log 2}}
{\cos((\pi/2)\sin d_{\epsilon})}
\right\}^{\beta}.
\end{align*}
Combining this estimate with the limits of the second and fourth
terms shown above, using~\eqref{eq:continuous-kappa} and the continuity
$\kappa_{L,d_{\epsilon}}\to\cos a_{L,d}$, and then letting
$\epsilon\to0$, we obtain
\[
\mathcal{N}_1(F,d)
 \leq \frac{2C_{\mathrm{D}}}{\pi\cos d}.
\]
\end{proof}

\subsection{Bound of the truncation error (proof of Lemma~\ref{lem:truncation-error})}
\label{subsec:truncation-error}

The following lemma is useful for
the proof of Lemma~\ref{lem:truncation-error}.

\begin{lemma}[Okayama et al.~{\cite[Lemma~4.7]{OkaShinKatsu}}]
For all $t\in\mathbb{R}$, we have
\begin{equation}
 \left|\frac{\log(1+\E^t)}{1+\log(1+\E^t)}\cdot\frac{1+\E^t}{\E^t}\right|
\leq 1.
\label{eq:bound-log-by-exp-real}
\end{equation}
\end{lemma}

Using this inequality, we prove
Lemma~\ref{lem:truncation-error} as follows.

\begin{proof}
From~\eqref{eq:f-bound-plus} and~\eqref{eq:domDdplus} with $\zeta=x$,
for all $x\geq 0$, we have
\[
 |F(x)|
\leq K_{+}\left(\E^{-\phi_5(x)}\right)^{\beta}
= K_{+}\left(\E^{1/\log(1+\E^{\pi\sinh x})}\right)^{\beta}
\frac{1}{(1+\E^{\pi\sinh x})^{\beta}}
\leq K_{+}\left(\E^{1/\log 2}\right)^{\beta}
\E^{-\pi\beta\sinh x},
\]
similar to~\eqref{eq:bound-F-zeta-plus}.
Using this estimate and $|S(k,h)(x)|\leq 1$, we have
\begin{align*}
 \left|\sum_{k=N+1}^{\infty}F(kh)S(k,h)(x)\right|
&\leq \frac{K_{+}\left(\E^{1/\log 2}\right)^{\beta}}{h}\cdot
h\sum_{k=N+1}^{\infty} \E^{-\pi\beta\sinh(kh)}\\
&\leq \frac{K_{+}\left(\E^{1/\log 2}\right)^{\beta}}{h}
\int_{Nh}^{\infty}\E^{-\pi\beta\sinh x}\D x\\
&\leq\frac{K_{+}\left(\E^{1/\log 2}\right)^{\beta}}{h\cosh(Nh)}
\int_{Nh}^{\infty}\cosh(x)\E^{-\pi\beta\sinh x}\D x\\
&=\frac{K_{+}\left(\E^{1/\log 2}\right)^{\beta}}{h\cosh(Nh)}
\cdot\frac{\E^{-\pi\beta\sinh(Nh)}}{\pi\beta}\\
&\leq \frac{K_{+}\left(\E^{1/\log 2}\right)^{\beta}}{h(\exp(Nh)/2)}
\cdot\frac{\E^{(\pi/2)\beta}\E^{-(\pi/2)\beta\exp(Nh)}}{\pi\beta}\\
&= \frac{K_{+}\left(\E^{(1/\log 2)+(\pi/2)}\right)^{\beta}}
{h (\pi/2)\beta\exp(Nh)} \E^{-(\pi/2)\beta\exp(Nh)}.
\end{align*}
Set $q=2dn/\mu$.  The lower bound on $n$ implies $q\geq\E$, and hence
\[
 \frac{\E^{-(\pi/2)\mu q\{1-1/\log q\}}}{\log q}\leq
  \frac{\E^{-(\pi/2)\mu \E\{1-1/\log \E\}}}{\log \E}=1.
\]
Furthermore, using~\eqref{eq:Def-MN-DE} and~\eqref{eq:Def-h-DE},
we have
\begin{align*}
 \frac{K_{+}\left(\E^{(1/\log 2)+(\pi/2)}\right)^{\beta}}
{h (\pi/2)\beta\exp(Nh)} \E^{-(\pi/2)\beta\exp(Nh)}
&\leq \frac{K_{+}\left(\E^{(1/\log 2)+(\pi/2)}\right)^{\beta}}
{\{\log(2 d n/\mu) / n\} \pi d n} \E^{-\pi d n}\\
&=\frac{K_{+}\left(\E^{(1/\log 2)+(\pi/2)}\right)^{\beta}}{\pi d}
\cdot\frac{\E^{-(\pi/2)\mu (2 d n/\mu)\left\{1 - 1/\log(2 d n/\mu)\right\}}}
{\log(2 d n/\mu)}\E^{-\pi d n/\log(2 d n/\mu)}\\
&\leq\frac{K_{+}\left(\E^{(1/\log 2)+(\pi/2)}\right)^{\beta}}{\pi d}
\E^{-\pi d n/\log(2 d n/\mu)}.
\end{align*}
On the other hand,
from~\eqref{eq:f-bound-minus},~\eqref{eq:bound-by-1-minus-log2}
with $\zeta=x$,
and~\eqref{eq:bound-log-by-exp-real} with $t=\pi\sinh x$,
for all $x < 0$, we have
\begin{align*}
 |F(x)|
&\leq K_{-}\frac{1}{\left|\phi_{5}(x)\right|^{\alpha}}\\
& = K_{-}\frac{1}{|- 1 + \log(1 + \E^{\pi\sinh x})|^{\alpha}}
\left\{
\frac{\log(1 + \E^{\pi\sinh x})}{1+\log(1 + \E^{\pi\sinh x})}
\right\}^{\alpha}\\
&\leq  \frac{K_{-}}{(1 - \log 2)^{\alpha}}
\cdot
\frac{1}{(1 + \E^{-\pi\sinh x})^{\alpha}}\\
&\leq \frac{K_{-}}{(1 - \log 2)^{\alpha}}\E^{\pi\alpha\sinh x}.
\end{align*}
Using this estimate and $|S(k,h)(x)|\leq 1$, we have
\begin{align*}
 \left|\sum_{k=-\infty}^{-M-1}F(kh)S(k,h)(x)\right|
&\leq \frac{K_{-}}{(1 - \log 2)^{\alpha} h}\cdot
h\sum_{k=-\infty}^{-M-1} \E^{\pi\alpha\sinh(kh)}\\
&\leq \frac{K_{-}}{(1 - \log 2)^{\alpha} h}
\int_{-\infty}^{-Mh}\E^{\pi\alpha\sinh x}\D x\\
&\leq\frac{K_{-}}{(1 - \log 2)^{\alpha} h \cosh(Mh)}
\int_{-\infty}^{-Mh}\cosh(x)\E^{\pi\alpha\sinh x}\D x\\
&=\frac{K_{-}}{(1 - \log 2)^{\alpha} h \cosh(Mh)}
\cdot\frac{\E^{-\pi\alpha\sinh(Mh)}}{\pi\alpha}\\
&\leq \frac{K_{-}}{(1 - \log 2)^{\alpha} h (\exp(Mh)/2)}
\cdot\frac{\E^{(\pi/2)\alpha}\E^{-(\pi/2)\alpha\exp(Mh)}}{\pi\alpha}\\
&=K_{-}\left(\frac{\E^{\pi/2}}{1 - \log 2}\right)^{\alpha}
\frac{\E^{-(\pi/2)\alpha\exp(Mh)}}{h (\pi/2)\alpha\exp(Mh)}.
\end{align*}
Furthermore, using~\eqref{eq:Def-MN-DE} and~\eqref{eq:Def-h-DE},
we have
\begin{align*}
 K_{-}\left(\frac{\E^{\pi/2}}{1 - \log 2}\right)^{\alpha}
\frac{\E^{-(\pi/2)\alpha\exp(Mh)}}{h (\pi/2)\alpha\exp(Mh)}
&\leq K_{-}\left(\frac{\E^{\pi/2}}{1 - \log 2}\right)^{\alpha}
\frac{\E^{-\pi d n}}{\{\log(2 d n/\mu)/n\} \pi d n}\\
& = \frac{K_{-}}{\pi d}\left(\frac{\E^{\pi/2}}{1 - \log 2}\right)^{\alpha}
\frac{\E^{-(\pi/2)\mu (2 d n/\mu)\left\{1 - 1/\log(2 d n/\mu)\right\}}}
{\log(2 d n/\mu)} \E^{-\pi d n/\log(2 d n/\mu)}\\
&\leq\frac{K_{-}}{\pi d}\left(\frac{\E^{\pi/2}}{1 - \log 2}\right)^{\alpha}
\E^{-\pi d n/\log(2 d n/\mu)}.
\end{align*}
Summing up the above estimates, we obtain the desired inequality.
\end{proof}

\section{Concluding remarks}
\label{sec:conclusion}

This paper has developed a DE-Sinc approximation for unilateral rapidly
decreasing functions based on the transformation $t=\phi_5(x)$.
Under the stated analyticity and decay assumptions, we established an
almost-exponential error estimate of order
$\OO(\exp(-cn/\log n))$ with a constant explicitly expressed in terms
of the problem parameters.  The resulting bound is computable and can
therefore be used for approximation with guaranteed accuracy.
Numerical examples satisfying the assumptions confirmed both the faster
convergence of the DE-Sinc approximation compared with the corresponding
SE-Sinc approximations and the validity of the derived error bound.

The additional example $f_3$ in~\eqref{eq:f_3} showed that the DE-Sinc
approximation may still be competitive even when the analyticity
assumption of the present theorem is not satisfied.  In that case, however, the
characteristic almost-exponential convergence was not observed, and the
current error bound is not applicable.  Establishing an error analysis
under weaker analyticity conditions is therefore an important topic for
future work.  Another direction is to apply the present approximation
and its computable error bound to Sinc numerical methods for differential
and integral equations whose solutions exhibit unilateral rapid decay.

\section*{Declaration of generative AI and AI-assisted technologies in the manuscript preparation process}

During the preparation of this work, the author used OpenAI Codex to
improve the language, clarity, and organization of the manuscript and
to assist in checking the mathematical exposition.  After using this
tool, the author reviewed and edited the content as needed and takes full
responsibility for the content of the publication.

\bibliography{DE-Sinc-unilateral}

\end{document}